\input ./style/arxiv-general.cfg
\documentclass[aap,MSNbibl,dvips]{arximspdf}
\makeatletter
   \@ifpackageloaded{graphicx}{}{\usepackage{graphicx}}
\makeatother

\doi{10.1214/14-AAP1057} 
\volume{25}
\issue{5}
\pubyear{2015}
\firstpage{2626}
\lastpage{2670}
\docsubty{FLA}

\makeatletter
\newcommand{\rrvert}{\vert}

\newcommand{\rrVert}{\Vert}
\newcommand{\llvert}{\vert}
\newcommand{\llVert}{\Vert}
\renewcommand{\mid}{|}
\newtheorem{proposition}{Proposition}
\newtheorem{theorem}{Theorem}
\newtheorem{corollary}{Corollary}
\newproclaim{rem}{Remark}
\def\N{{\mathbb N}}
\def\R{{\mathbb R}}
\def\Z{{\mathbb Z}}
\def\P{{\mathbb P}}
\def\E{{\mathbb E}}

\def\Var{\operatorname{Var}}
\def\diff{\mathrm{d}}
\makeatother

\begin{document}
\begin{frontmatter}

\title{A stochastic analysis of resource sharing with logarithmic weights}
\runtitle{Bandwidth sharing algorithm}

\begin{aug}
\author[A]{\fnms{Philippe}~\snm{Robert}\corref{}\ead[label=e1]{Philippe.Robert@inria.fr}}
\and
\author[B]{\fnms{Amandine}~\snm{V\'eber}\ead[label=e2]{Amandine.Veber@cmap.polytechnique.fr}}
\runauthor{P. Robert and A. V\'eber}
\affiliation{INRIA Paris---Rocquencourt and \'{E}cole Polytechnique}
\address[A]{INRIA Paris---Rocquencourt\\
Domaine de Voluceau\\
78153 Le Chesnay\\
France\\
\printead{e1}} 
\address[B]{Centre de Math\'{e}matiques Appliqu\'{e}es\\
\'{E}cole Polytechnique\\
Route de Saclay\\
91128 Palaiseau Cedex\\
France\\
\printead{e2}}
\end{aug}

\received{\smonth{10} \syear{2013}}
\revised{\smonth{7} \syear{2014}}

%
\begin{abstract}
The paper investigates the properties of a class of resource allocation
algorithms for communication networks: if a node of this network has
$x$ requests to transmit, then it receives a fraction of the capacity
proportional to $\log(1+x)$, the logarithm of its current load. A
detailed fluid scaling analysis of such a network with two nodes is
presented. It is shown that the interaction of several time scales
plays an important role in the evolution of such a system, in
particular its coordinates may live on very different time and space
scales. As a consequence, the associated stochastic processes turn out
to have unusual scaling behaviors. A heavy traffic limit theorem for
the invariant distribution is also proved. Finally, we present a
generalization to the resource sharing algorithm for which the $\log$
function is replaced by an increasing function. Possible
generalizations of these results with $J>2$ nodes or with the function
$\log$ replaced by another slowly increasing function are discussed.
\end{abstract}

%
\end{frontmatter}

\section{Introduction}
The resource allocation problem considered in this paper involves $J$
nodes which have access to a common shared resource, for example, a~communication channel or a processing unit. The resource is assumed to
have a fixed capacity, say $1$. The resource is shared among nodes in
the following way: for $1\leq j\leq J$, if node $j$ has $n_j$ requests
pending, it receives the instantaneous fraction of capacity
%
\begin{equation}
\label{algog} \frac{f(n_j)}{f(n_1)+ f(n_2)+\cdots+f(n_J)}
\end{equation}
from the resource. The algorithm is thus defined by the function
$x\mapsto f(x)$. There are several situations where the capacity is
allocated in this way. It should be noted that our results are proved
in the case where $J=2$. The general case $J\geq2$ is briefly
sketched. See Section~\ref{Multisec} for the conjectured behavior of
this system.

\subsection{Saturated node of the Internet} In this context, the nodes
correspond to TCP flows with different sources and destinations. The
resource here is the processing time of a fixed router on the path of
these flows. The packets of a flow are queued in the buffer of the
router until they are routed to the next stage of their path.
Congestion is simply the situation when the buffer is full and,
therefore, incoming packets are lost. Because of the TCP protocol, a
given flow will increase or decrease the rate at which it sends the
packets, depending on the level of congestion of the routers on its
path. There are several ways to represent this phenomenon. It should be
kept in mind that the following descriptions are mathematical models of
the way TCP is \emph{thought} to allocate bandwidth, not of the TCP
algorithm itself. See Massouli\'e and Roberts~\cite{MR}.
\begin{longlist}[(a)]
\item[(a) \textit{Processor-sharing disciplines}.]
A popular, simplified, stochastic model of this situation consists in
considering that the router allocates its processing power to each flow
according to a slight generalization of the allocation policy given by
relation~(\ref{algog}) with a function $f$ depending on the node $j$
and of the form $w_j n$, where $1/w_j$ can be the round trip time
between the source and the destination. This allocation algorithm
corresponds to the \emph{discriminatory processor-sharing policy}. Node
$j$ has an instantaneous fraction of capacity given by
%
\begin{equation}
\label{dps} \frac{w_jn_j}{w_1n_1+w_2n_2+\cdots+w_Jn_J}.
\end{equation}
See Altman et~al. \cite{AAU} and references therein for a survey.
When all the $w_j$'s are~$1$, we obtain the classical processor-sharing
policy: node $j$ receives the fraction of capacity $n_j/(n_1+\cdots
+n_J)$, and the bandwidth is equally divided among the current
requests. Different classes of stochastic models of processor-sharing
policies have been extensively used to describe the congestion in IP
networks. See Bonald et~al. \cite{Bonald2}, Kelly et~al. \cite
{Kelly} and Graham and Robert \cite{Graham} and references therein.
\end{longlist}

\begin{longlist}[(b)]
\item[(b) \textit{Alpha-fair disciplines}.]
These policies have also been introduced to describe the allocation of
bandwidth in IP networks (see Mo and Walrand \cite{Walrand}), in terms
of an optimization problem (cf. Kelly et~al. \cite{Kelly}). In our
context, a related policy would correspond to the case $f(n)=n^\alpha
$, $n\geq0$, so that a nonempty node $j$ has an instantaneous fraction
of capacity given by
%
\begin{equation}
\label{alp} \frac{n_j^\alpha}{n_1^\alpha+n_2^\alpha+\cdots+n_J^\alpha}.
\end{equation}
The case $\alpha=1$ is the processor-sharing discipline presented above.
\end{longlist}
In the wireless section below, the situation is quite different since
the bandwidth allocation algorithm is defined explicitly by relations
similar to~(\ref{algog}).

\subsection{Wireless networks} This is again a simplified, but
meaningful\break stochastic model of bandwidth allocation, this time in
wireless networks. The resource here is a radio channel in a region
where there are $J$ stations/mobiles. At a given time, because of
interferences, only one station can transmit successfully in this
region. A station with $n_j$ messages waiting for transmission can
detect if there is a communication going on or not. If not, a classical
backoff mechanism is used: the station starts transmitting after an
exponentially distributed amount of service with parameter $f(n_j)$. If
another station starts a transmission before that time, the attempt of
transmission is canceled. Consequently, if initially there is no
transmission, node $j$ will be the first to access the channel with probability
${f(n_j)}/{(f(n_1)+\cdots+f(n_J))}$.
See Abramson~\cite{Abramson} and Metcalf and Boggs~\cite{Metcalf} for
historical references, and Tassiulas and Ephremides~\cite{Tass}. If a
small quantum $\delta$ is transmitted at each access, it is not
difficult to see that, provided that backoff times are small (which is
the case if one of the components is large), as $\delta$ goes to $0$
the effective capacity allocated to station $j$ is indeed given by
%
\begin{equation}
\label{eef} \frac{f(n_j)}{f(n_1)+\cdots+f(n_J)}.
\end{equation}
This is the analogue of the approximation of the round robin policy by
the processor-sharing discipline.

\subsubsection*{Fair access to resource: The choice of the function $x\mapsto f(x)$}
The function $f$ should clearly be increasing, so that the fraction of
the capacity allocated may grow with the number of requests. This is
the case if $f(x)=x^\alpha$ which corresponds to the Alpha-fair
disciplines already mentioned. However, these policies may have a
serious drawback. Indeed, if a station $j$ has a large number of
requests pending while the other stations are lightly loaded, the
latter will receive a negligible fraction of the bandwidth. The station
$j$ will therefore capture the channel for its own benefit, until the
instant when some of the other stations reach a \emph{comparable level
of congestion}. This is a highly undesirable property for a network
where fairness issues (for nodes, not requests) are of primary
importance. See Bonald and Massouli\'e~\cite{Bonald}.

A possible way of solving this problem is to consider increasing
functions $f$ which grow slowly to infinity like, for example, the
concave function $x\mapsto\log(1+x)$ or $x\mapsto\log\log(e+x)$.
In this way, one can expect to reduce significantly the impact of
saturated nodes even if they still receive a sizable fraction of the
available capacity. Related algorithms have been considered in the
context of wireless networks; see Shah and Wischik~\cite{Shah} and
references therein. Bouman et~al. \cite{BBLP} and Ghaderi et~al. \cite
{GBW} investigate the impact of the growth of the function~$f$ on the
stability and on the delays for several wireless network architectures
with a related bandwidth allocation scheme.

In this paper, we mainly investigate the case $f(x)=\log(1+x)$. The
general case is sketched in Section~\ref{Genfsec}. The instantaneous
fraction of capacity of the $j$th node is therefore given by
%
\begin{equation}
\label{logps} \frac{\log(1+n_j)}{\log(1+ n_1)+\log(1+n_2)+\cdots+\log(1+n_J)}.
\end{equation}
Note also that since the limit of $x^\alpha/\alpha$ when $\alpha$
goes to $0$ is $\log x$ for $x>0$, this allocation mechanism can be
seen as a limiting case of Alpha-fair disciplines described above.

The $\log$ function moderates the rate at which a saturated station
tries to access the resource, which is a desirable property in an
heterogeneous network where connections may have very different
characteristics. In the context of wireless networks, a related
algorithm was used to show that an optimal stability region is possible
in a quite general network. The growth properties of the $\log$
function play an important role in the proof of the result. Basically,
the $\log$ of the states of the saturated stations being quite stable
on some large time intervals, the schedule (the set of stations that
can transmit at some time) quickly reaches some equilibrium and stays
around it. A Lyapounov function argument can then be used to prove
ergodicity (see Shah and Shin~\cite{Shah2}). Up to now, apart from
these stability results, little is known about the quantitative and
qualitative properties of these algorithms. As we shall see below, the
mathematical analysis of this class of algorithms presents some
challenging and unusual problems (see also Wischik~\cite{Wischik}). We
first achieve a fluid limit scaling analysis, which gives a very
precise description of the qualitative behavior of these algorithms.
Additionally, we derive a heavy traffic limit theorem result for the
invariant distribution of the associated Markov process. Before
presenting our main results, we briefly recall the main definitions of
the fluid limit scaling. The interested reader will find an extended
presentation in Bramson~\cite{Bramson} or in Chapter~8 of Robert~\cite
{Robert}.

\subsection*{Fluid limits}
Throughout the paper, it will be assumed that, for every \mbox{$1\leq j\leq
J$}, the requests arriving at the $j$th node form a Poisson process with
rate $\lambda_j>0$. Each request at node~$j$ leaves the network when
it has received an exponentially distributed amount of service with
parameter $\mu_j$ from the common resource. The average load of the
$j$th node is denoted by $\rho_j=\lambda_j/\mu_j$.

The fluid limit scaling of a stochastic process $(Z(t))$ in $\R^J$
consists in speeding up time and space in proportion to the norm of its
initial state:
\[
\overline{Z}_N(t)=\frac{Z(Nt)}{N}\qquad\mbox{with } N=\bigl\llVert
Z(0)\bigr\rrVert.
\]
A possible limit in distribution of the sequence of processes
$(\overline{Z}_N(t))$ is called a fluid limit of the process $(Z(t))$.
Hence, in some sense fluid limits give a first-order description of
$(Z(t))$. This is a convenient tool to investigate multidimensional
processes for which general results are scarce. In the context of
Markov processes, there is an additional interest since the ergodicity
of the process can be connected to the fact that fluid limits (whose
initial states lie on the unit sphere) return to the origin. See Rybko
and Stolyar~\cite{Rybko} and Dai~\cite{Dai}. Note, however, that the
fluid limit scaling is well suited for processes that behave locally
like random walks. For other processes, different scalings may have to
be considered.

For certain choices of $f$ a fluid limit is obtained by standard
techniques. For instance, in the case
of generalized processor-sharing policy defined by relation~(\ref{dps}), for every $1\leq j\leq J$ and $t\geq0$, let $X_j(t)$ denote
the number of jobs waiting at the $j$th node. The evolution of this
process can be represented as
\[
X_j(t)= X_j(0)+M_j(t)+
\lambda_j t-\mu_j\int_0^t
\frac{w_j
X_j(s)}{w_1X_1(s)+ w_2X_2(s)+\cdots+w_JX_J(s)}\, \diff s,
\]
where $(M_j(t))$ is a martingale. The scaled process with $N=\llVert  X(0)\rrVert  $
is thus given by
%
\begin{eqnarray}
\label{rou} \overline{X}{}^N_j(t)&=&\overline{X}{}^N_j(0)+
\frac{M_j(Nt)}{N}
\nonumber\\[-8pt]\\[-8pt]\nonumber
&&{}+\lambda _j t-\mu_j\int_0^t
\frac{w_j \overline{X}{}^N_j(s)}{w_1\overline
{X}{}^N_1(s)+ w_2\overline{X}{}^N_2(s)+\cdots+w_J\overline{X}{}^N_J(s)}\,\diff s.
\end{eqnarray}
A standard tightness criterion and the fact that the martingale
\[
\bigl(\bigl({M_j(Nt)}/{N}\bigr), 1\leq j\le J\bigr)
\]
converges in distribution to $0$ imply that any fluid limit $((x_j(t)),
1\leq j\leq J)$ should satisfy the ordinary differential equations
%
\begin{equation}
\label{fluidps} \quad\frac{\diff x_j}{\diff t}(t)= \lambda_j-\mu_j
\frac{w_j
x_j(t)}{w_1x_1(t)+ w_2x_2(t)+\cdots+w_Jx_J(t)},\qquad1\leq j\leq J,
\end{equation}
on the interval $[0,t_0]$, provided that the vector $(x_j(t), 1\leq
j\leq J)$ does not hit~$0$ before $t_0$. See Ben Tahar and
Jean-Marie~\cite{BenTahar}, Ramanan and Reiman~\cite{RR} and
references therein.

Similarly, for Alpha-fair disciplines the corresponding fluid model
$(x_j(t))$ is the solution to the ODE
%
\begin{equation}
\label{fluidalpha} \frac{\diff x_j}{\diff t}(t)= \lambda_j-\mu_j
\frac{x_j(t)^\alpha
}{x_1(t)^\alpha+ x_2(t)^\alpha+\cdots+x_J(t)^\alpha}, \qquad1\leq j\leq J.
\end{equation}
Observe that, for these two choices of function $f$, the scaled process
satisfies an autonomous ODE, like~(\ref{fluidalpha}), with a
stochastic noise component that vanishes as $N$ gets large.
Secondly, a remarkable feature of these convergence results is that all
coordinates of the scaled process are of order $N$. That is, as long as
$x(t)$ is not the vector $0$, one has $x_j(t)>0$ \emph{for all} $1\leq
j\leq J$. However, as we shall now discuss, for instances, of interest
where $f$ increases slowly [e.g., $f(x)=\log x$] such standard
techniques are not applicable.

\subsection*{Problem of fluid limits for algorithms with logarithmic weights}
Let\break $(L_j(t), 1\leq j\leq J)$ denote the Markov process associated to
the policy with logarithmic weights, that is, associated to
relation~(\ref{logps}). Due to the $\log$ function, the scaled
process $(\overline{L}_j(t))$ does not have an autonomous
representation analogous to~(\ref{rou}). In fact, it is easy to see
that the corresponding stochastic equations involve both $L_j(Nt)$ and
$\log(1+L_j(Nt))$, two quantities which evolve on very different
scales. For this reason, there is no way of guessing a system of
plausible ``fluid equations'' corresponding to system~(\ref{fluidps})
[resp., to~(\ref{fluidalpha})] for discriminatory processor-sharing
policy (resp., Alpha-fair policies). See Wischik~\cite{Wischik} and
Ghaderi et~al. \cite{GBW}. Note that the question of stability of the
system is not an issue here. Indeed, because of the work conserving
property of these policies, a necessary and sufficient condition for
the ergodicity of $(L_j(t), 1\leq j\leq J)$ is simply given by
\[
\rho_1+\cdots+\rho_J<1\qquad\mbox{with }
\rho_j=\frac{\lambda
_j}{\mu_j}, 1\leq j\leq J.
\]
To the best of our knowledge, there is no explicit expression known for
the invariant distribution.
The fluid scaling gives a first order description of the behavior of
this policy. As we shall see, an interesting convergence result for the
invariant distribution just below saturation can also be derived from
these results.

Additionally, an interesting phenomenon in this domain is presented.
For most of the queueing networks investigated up to now, the classic
general scheme for the fluid scaling of the associated Markov process
$(X(t))$ is as follows: there is a subset of the coordinates whose
values are of the order of $N=\llVert  X(0)\rrVert  $ and the other coordinates form
an ergodic Markov process whose invariant distribution determines the
evolution of the large coordinates on the fluid scale. Here, the
situation is different. In the case of two nodes, under some
appropriate conditions, then the coordinate $(L_2(t))$ is of the order
of $N$ but $(L_1(t))$ is an order of magnitude smaller, $N^{\alpha^*}$
for some $0<{\alpha^*}<1$. There is indeed an underlying ergodic Markov
process but at the second order, namely an Ornstein--Uhlenbeck process
$(Z(t))$ so that $L_1(t)\sim N^{\alpha^*}+\sqrt{N^{{\alpha^*}}\log
N}\cdot
Z(t)$. For this queueing system, the queue lengths of underloaded
queues are not proportional to load but operate on a scale between
$O(1)$ and $O(N)$.

\subsection*{Outline of the paper}
Section~\ref{Pressec} presents the main results of the paper.
Section~\ref{stocmod} introduces the notation and the stochastic
differential equations associated to the Markov process $(L_j(t), 1\leq
j\leq J)$. Sections~\ref{ts1}, \ref{ts2} and~\ref{ts3} are,
respectively, devoted to the scaling properties of the time scales
$t\mapsto N^t$, $t\mapsto N^{\alpha^*}\log N t$ and $t\mapsto N t$. There
we provide a precise description of the evolution of the network,
together with some estimates of hitting times. The key results on the
fluid limits are presented in Section~\ref{ts3}.
In Section~\ref{heavysec}, we prove a heavy traffic limit theorem for
the invariant distribution. The case of a two node network with a
general $f$ is discussed in Section~\ref{loglogsec}. The corresponding
time scales are identified in this case. Section~\ref{Multisec} gives
a brief sketch of the case of a network of $J$ nodes.

\section{Presentation of results}\label{Pressec}

\subsection{The general picture}\label{TGP}
The main result of the paper is that the states of the nodes of a
network with two nodes may live on different space and time scales
depending on the set of arrival and service rates. In the next
paragraph, we give a more precise description of this phenomenon. An
illustration of the different scales is given in Figure~\ref{fig1},
where the $y$-axis is on a $\log\cdot/\log N$ scale.

\begin{figure}

\includegraphics{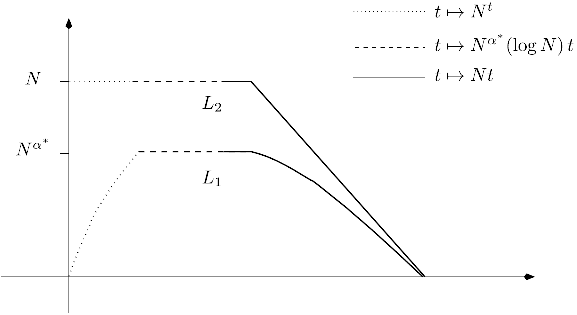}

\caption{A first-order picture of the network with $\rho_1+\rho
_2<1$, $\rho_1<1/2$ and $(L_1(0),L_2(0))=(0,N)$.
The first queue grows according to step \textup{(a)}, then the queue behaves as
an OU process for a while according to step \textup{(b)}, then finally the system
converges to zero according to step \textup{(c)}.}\label{fig1}
\end{figure}

Giving an asymptotic picture of this queueing system with only two
nodes is already a challenging problem. To concentrate on the most
interesting case, from the point of view of mathematical difficulty,
let us assume that the parameters satisfy the conditions:\vspace*{1pt} $\rho_1<1/2$
and $\rho_2>1/2$ and the initial state of the process is
$(L_1^N(0),L_2^N(0))=(0,N)$. The other cases are discussed further in
the paper, in Proposition~\ref{fluid+}, for example.

\subsubsection*{Three time scales}
\begin{longlist}[(a)]
\item[(a) \emph{The time scale} ${t\to N^t}$.]
A convergence result, Proposition~\ref{initprop}, shows the
convergence in distribution, for any $0<s_0<t_0<{\alpha^*}$,
\[
\lim_{N\to+\infty} \bigl(L_1^N
\bigl(N^t\bigr), u<t<v \bigr)= \biggl( \biggl(\lambda_1-
\mu_1\frac{t}{t+1} \biggr) N^t, s_0<t<t_0
\biggr)
\]
with
\[
{\alpha^*}\stackrel{\mathrm{def.}} {=}\frac{\rho_1}{1-\rho_1}.
\]
The process $L_2^N$ stays at $N$ on this time scale. The condition
$\rho_1<1/2$ implies that ${\alpha^*}<1$. Note that the prefactor of
$N^t$ vanishes at $t={\alpha^*}$. For this reason, this convergence
result does not prove that values of the order of $N^{\alpha^*}$ can be
reached. In fact, an extra $(\log N)$ factor is required and
Proposition~\ref{Job} shows that if $\delta<1$, the average hitting
time of the value $\lfloor\delta N^{\alpha^*}\rfloor$ by $(L_1^N(t))$
is bounded above by $K_1 N^{\alpha^*}\log N$ for some $K_1>0$. On the
other hand, reaching the value $\lfloor N^{\alpha^*}\rfloor$ is slightly
longer: the average hitting time of $\lfloor N^{\alpha^*}\rfloor$ is
bounded by $C N^{\alpha^*}(\log N)^2\log\log(N)$ for some $C>0$; see
relation~(\ref{Goisot}) in Section~\ref{ts2}.
\end{longlist}

\begin{longlist}[(b)]
\item[(b) \emph{The time scale} ${t\to N^{\alpha^*}(\log N) t}$.]
We now assume that $L_1^N(0)=\lfloor N^{\alpha^*}\rfloor$ and
$L_2(0)=N$. Theorem~\ref{theostab} proves the following convergence in
distribution
%
\begin{equation}
\label{tgt} \lim_{N\to+\infty} \biggl(\frac{L_1^N (N^{\alpha^*}(\log N)
t
)-N^{\alpha^*}}{\sqrt{N^{{\alpha^*}}\log N}}, t\geq0
\biggr)= \bigl(Z(t)\bigr),
\end{equation}
where $(Z(t))$ is an Ornstein--Uhlenbeck process. In other words, on
this time scale $L_1^N$ is stabilized around the value $N^{\alpha^*}$.
Again, the process $L_2^N$ stays at $N$ on this time scale.
\end{longlist}

\begin{longlist}[(c)]
\item[(c) \emph{The fluid time scale} ${t\to Nt}$.]
Theorem~\ref{theoprov} shows the convergence in distribution,
%
\begin{equation}
\label{sap} \lim_{N\to+\infty} \biggl(\frac{L_1^N(Nt)}{N^{\alpha^*}},
\frac
{L_2^N(Nt)}{N} \biggr)= \bigl(\gamma(t)^{\alpha^*},\gamma(t) \bigr)
\end{equation}
for the convergence in distribution of processes, with
\[
\gamma(t)=\bigl(1+\bigl(\lambda_2-\mu_2(1-
\rho_1)\bigr)t\bigr)^+.
\]
Consequently, as long as the fluid limit of $(L_2(t))$ is not $0$, the
process $L_1^N$ lives on the space scale $N^{\alpha^*}$.
\end{longlist}


\subsection{Properties of resource sharing with logarithmic weights}
The bandwidth allocation with logarithmic weights exhibits some
interesting properties.
For the two node network described above, when the initial state is
$(0,N)$ we prove that the fluid limit is given by
\[
\bigl( \bigl(0, 1+\bigl(\lambda_2-\mu_2(1-
\rho_1)\bigr)t \bigr)^+ \bigr).
\]
This shows that node~2 receives the capacity $1-\rho_1$, which is
another way of saying that node~1 is stable at the fluid level. The
simplicity of this expression somewhat hides the complexity of the
situation, since the quantity ${\alpha^*}$ does not show up. Yet, the
quantity ${\alpha^*}$ has a crucial impact on the equilibrium
distribution and on the transient behavior.

\subsubsection*{Equilibrium: Heavy-traffic regime}
Under the stability condition $\rho_1+\break \rho_2<1$, let $(L_{1,\rho
},L_{2,\rho})$ denote random variables with the equilibrium
distribution of the Markov process $(L_1(t),L_2(t))$. If $\rho_1<1/2$
is fixed, the heavy traffic-limit Theorem~\ref{heavyth} shows the
convergence in distribution
\[
\lim_{\rho_2\nearrow1-\rho_1} \bigl((1-\rho_1-\rho_2)^{{\alpha^*}
}L_{1,\rho},
(1-\rho_1-\rho_2)L_{2,\rho} \bigr)=
\bigl(X^{\alpha^*},X\bigr),
\]
where $X$ is an exponential random variable. Hence, at equilibrium and
in the heavy traffic regime, the relation $L_{1}\sim L_{2}^{\alpha^*}$
also holds as in relation~(\ref{sap}) for fluid limits. Note that
$X^{\alpha^*}$ has a Weibull distribution.

\subsubsection*{Transient case}
If the system is overloaded ($\rho_1+\rho_2>1$) and if $\rho_1<1/2$,
the size of the queue of class~1 requests grows at rate proportional to
$t^{\alpha^*}$, with ${\alpha^*}<1$. This implies that queue~1 is
stable at
the fluid level, that is, that $L_1(t)/t$ goes to $0$ in distribution
as $t$ becomes large. Hence, without any priority mechanism among
nodes, if a node has a light load, $\rho_1<1/2$, then most of its
messages will be transmitted with success even in the case where the
system is globally saturated. Recall that in the transient case of the
processor-sharing policy or even with the Alpha-fair disciplines, this
is not true at all: the states of the nodes diverge to infinity at the
same speed, linearly in time.

This is an interesting feature from the point of view of fairness
issues among nodes. Indeed, if the node is not too aggressive, $\rho
_1<1/2$, this result implies that it will be able to transmit most of
its traffic \emph{independently} of the load of the other node. It can
be shown that an analogous property is valid for the network with $J$
nodes; see Section~\ref{Multisec}.

It is unlikely that a standard fluid analysis, that is, deriving
directly some equations similar to relation~(\ref{fluidps}), for
example, can be done to investigate the qualitative behavior of more
complex networks. This is where the consideration of the various time
scales is useful. It gives a tool to explain, via a dynamic picture,
the multiple orders of magnitude of the state variables at equilibrium.

\subsection{An interaction of time scales}
There is an unconventional property for the fluid scaling of a queueing
system. For most of the queueing networks investigated up to now, the
classic general scheme for the fluid scaling of the associated Markov
process $(X(t))$ is as follows: there is a subset of the coordinates
whose values are of the order of $N=\llVert  X(0)\rrVert  $ and the other coordinates
form an ergodic Markov process whose invariant distribution determines
the evolution of the large coordinates on the fluid scale. See
Malyshev~\cite{Malyshev} and Bramson~\cite{Bramson} for some examples.

Here the situation is different. If the initial state is $(0,N)$ and if
$\rho_1<1/2$, then the large coordinate $L_2$ is of the order of $N$
but $L_1$ is an order of magnitude smaller, $N^{\alpha^*}$ with
$0<{\alpha^*}<1$. There is indeed an underlying ergodic Markov process
but at the second order, namely an Ornstein--Uhlenbeck process scaled
by a factor $\sqrt{N^{{\alpha^*}}\log N}$. See relation~(\ref{tgt}).

The associated stochastic model exhibits a \emph{stochastic averaging
principle} at the origin of the second expansion in relation~(\ref
{sap}). The key technical result of the paper, Theorem~\ref{lem1},
states that when $\rho_1<1/2$, on \emph{the fluid time scale},
$t\mapsto Nt$, $L_1^N$ is \emph{uniformly} of the order of
$(L_2^N)^{\alpha^*}$ on any finite time interval with high probability.
Recall that:
\begin{longlist}[(a)]
\item[(a)] If $L_2^N(0)=N$ and $L_1^N(0)=N^{\alpha^*}$, then on the time scale
$t\mapsto N^{\alpha^*}(\log N) t$ we have $L_2^N\sim L_2^N(0)$ and
$L_1^N$ is of the order of $(L_2^N(0))^{\alpha^*}$. Additionally, $L_1^N$
can be represented by an Ornstein--Uhlenbeck process around $N^{\alpha^*}
$; see previous point, point (b) in Section~\ref{TGP}.
\item[(b)] On the fluid time scale, $L_2^N(Nt)$ is of the order of $\gamma
(t)L_2^N(0)$ with $\gamma(t)$ defined above by relation~(\ref{sap}).
\end{longlist}
The problem lies in proving that on the fluid time scale $L_1^N$ adapts
sufficiently quickly to preserve the relation $L_1^N\sim
(L_2^N)^{\alpha^*}$. A central limit result, Proposition~\ref{tclgen},
suggests that this is not the case on the timescale of the
Ornstein--Uhlenbeck process, at least for a second-order description.
On the other hand, on the fluid time scale Theorem~\ref{lem1} shows
that this separation of time scales holds. Its proof uses several
estimates related to average hitting times of reflected random walks
and some coupling arguments. One of the problems encountered is that\vspace*{1pt}
the potential natural stochastic fluctuations of the fluid time scale,
of the order of $\sqrt{N}$, can be large compared to $N^{\alpha^*}$ (if
${\alpha^*}<1/2$, e.g.). In particular, standard stochastic
calculus cannot be used as such to prove the result. It turns out that
the potentially large fluctuations are reduced by the strong ergodicity
properties of the underlying Ornstein--Uhlenbeck process. Thus, it does
not seem that the classical techniques for proving stochastic averaging
results can be used here. See Has'minski{\u\i}~\cite{Khasminski01},
Freidlin and Wentzell~\cite{Freidlin02} and Papanicolau et~al. \cite
{PSV} for a general presentation of methods to prove stochastic
averaging principles.

\section{The stochastic model}\label{stocmod}
In this section, we introduce the main stochastic processes and some notation.
If $h$ is a nonnegative Borelian function on $\R_+$, we let
${\mathcal N}_h$ denote a Poisson process with rate $x\mapsto h(x)$ on
$\R_+$. This process can be defined as follows. If ${\mathcal P}$ is
a homogeneous Poisson point process on $\R_+^2$ with rate~$1$ and $f$
is some nonnegative Borelian function on $\R_+$, then ${\mathcal
N}_h(f)$ is defined by
\[
\int f(u){\mathcal N}_{h(u)}(\diff u)= \int_{\R_+^2}f(u){
\mathcal P}\bigl(\bigl[0,h(u)\bigr]\times\diff u\bigr).
\]
For $\xi\geq0$, ${\mathcal N}_\xi$ denotes the Poisson process with
rate $\xi$ on $\R_+$, that is, corresponding to the constant function
equal to $\xi$. In addition, for any $0\leq a\leq b$, ${\mathcal
N}_\xi([a,b])$ stands for the number of points of ${\mathcal N}_\xi$
in the interval $[a,b]$. Throughout the paper, the various Poisson
processes used will be assumed independent.

We consider two classes of customers. The arrival process of class $j$
customers is a Poisson process with rate $\lambda_j$, the distribution
of the duration of the required service is exponential with rate $\mu
_j$, and $\rho_j$ denotes the ratio $\lambda_j/\mu_j$. Each class of
customers has a dedicated queue and there is a single server working at
unit speed. If the state of the system is $(x_1,x_2)\in\N^2$, where
$x_j$ is the number of jobs in queue~$j$, then customers of class $j$
receive the fraction of service
%
\begin{equation}
\label{W} W_i(x_1,x_2)\stackrel{
\mathrm{def.}} {=}\frac{\log(1+x_i)}{\log
(1+x_1)+\log(1+x_2)}, \qquad i=1, 2,
\end{equation}
from the server, with the convention that $0/0$ is $0$. The process of
the number of jobs in queue $j\in\{1,2\}$ is denoted by $(L_j(t))$.
Since we are only interested in the total number of customers of each
class, there is no need to specify the service discipline for each
queue. It can be Processor-Sharing or FIFO (First In First Out), for example.

\subsection*{Stochastic differential equation}
The stochastic process $(L_1(t),L_2(t))$ can be expressed as the
solution to the following stochastic differential equation (SDE):
%
\begin{equation}
\diff L_i(t)={\mathcal N}_{\lambda_i}(\diff t) -{\mathcal
N}_{\mu_i
W_i(L_1(t-),L_2(t-))}(\diff t),\qquad i=1, 2, \label{SDE}
\end{equation}
where $L_i(t-)$ denotes the left limit of $L_i$ at $t$ and $W_i$ is the
function defined by relation~(\ref{W}).

\subsection*{A saturated system}
For $N\in\N$, $\lambda$, $\mu>0$, it will be convenient to
introduce a one-dimensional process $(X_N(t))$ describing the evolution
of the number of customers in a given queue when the number of jobs in
the other queue is ``large,'' that is, of the order of $N$. The process
$(X_N(t))$ is thus defined as the solution to the~SDE
%
\begin{equation}
\label{eqX} \diff X_N(t)={\mathcal N}_\lambda(\diff t) -{
\mathcal N}_{\mu
W_1(x,N-1)}(\diff t).
\end{equation}
From a Markov process point of view, $(X_N(t))$ is simply a \emph{birth
and death process} on $\N$ whose $Q$-matrix $(q(x,y))$ is defined by
\[
\cases{ q(x,x+1)=\lambda,
\cr
\displaystyle q(x,x-1) =\mu \frac{\log(1+x)}{\log(1+x)+\log
N},
\qquad x>0.}
\]
As we shall see, when $(L_1(0),L_2(0))= (0,N-1)$, $X_N(0)=0$, $\lambda
=\lambda_1$ and \mbox{$\mu=\mu_1$}, the two processes $(L_1(t))$ and
$(X_N(t))$ are close enough (for our purposes). Note that this is not
completely clear since the process $(L_2(t))$ may drift away from $N$
and therefore change the service rate received by each class. It turns
out that, because of the slow increase of the $\log$ function, this
property will hold at least at the beginning of the sample paths.

By integrating the SDE~(\ref{eqX}) one obtains that, for any $t\geq0$,
\begin{eqnarray}\label{eqXint}
X_N(t)&=&X_N(0)+{\mathcal N}_\lambda\bigl([0,t]
\bigr)-\int_0^t {\mathcal N}_{\mu W_1(X_N(u-),N-1)}(
\diff u)
\nonumber\\[-8pt]\\[-8pt]\nonumber
&=&X_N(0)+\lambda t-\mu\int_0^t
\frac{\log(1+X_N(u))}{\log
(1+X_N(u))+\log N}\,\diff u +M_N(t),
\nonumber
\end{eqnarray}
where $(M_N(t))$ is the martingale
\[
M_N(t)={\mathcal N}_\lambda\bigl([0,t]\bigr)-\lambda t +\int
_0^t \bigl[{\mathcal N}_{\mu W_1(X_N(u-),N-1)}(\diff
u)-\mu W_1\bigl(X_N(u),N-1\bigr)\,\diff u \bigr],
\]
whose increasing process is given by
%
\begin{equation}
\label{croc} \langle M_N \rangle(t)=\lambda t+ \mu\int
_0^t \frac
{\log(1+X_N(u))}{\log(1+X_N(u))+\log N}\,\diff u.
\end{equation}

\section{The initial phase}\label{ts1}
This section is devoted to the very beginning of the evolution of the
first component $(L_1^N(t))$, when it starts from $0$ while $L_2^N(0)=N$.
To start with, we have the following asymptotic result on the initial
growth rate of the process $(X_N(t))$ defined by equation~(\ref{eqX}).
Here and later, we write $a\wedge b$ for the quantity $\min(a,b)$.

\begin{proposition}\label{asympX}
If $X_N(0)=0$ and
\[
{\alpha^*}\stackrel{\mathit{def.}} {=}\frac{\rho}{1-\rho}\qquad \mbox {where }
\rho=\frac{\lambda}{\mu},
\]
then, for any $0<s_0<t_0< {\alpha^*}\wedge1$, the convergence in
distribution of stochastic processes
\[
\lim_{N\to+\infty} \biggl(\frac{X_N(N^t)}{N^t}, s_0\leq t
\leq t_0 \biggr)= \biggl(\lambda- \mu\frac{t}{t+1},
s_0\leq t\leq t_0 \biggr)
\]
holds.
\end{proposition}

See Chapters~2 and~3 of Billingsley~\cite{Billingsley} on the
convergence in distribution of a sequence of processes to a continuous
stochastic processes.
\begin{pf*}{Proof of Proposition \ref{asympX}}
The evolution equation~(\ref{eqXint}) and a change of variables give
us that for every $t\geq0$,
\begin{eqnarray*}
\frac{X_N(N^t)}{N^t}&=&\frac{X_N(1)-\lambda-M_N(1)}{N^t}+ \frac
{M_N(N^t)}{N^t}
\\
&&{}+\lambda-\mu\int_0^t \frac{\log(1+X_N(N^u))}{\log
(1+X_N(N^u))+\log N} (\log
N) N^{u-t}\,\diff u.
\end{eqnarray*}
Letting $Z_N(t)\stackrel{\mathrm{def.}}{=}(1+X_N(N^t))/N^t$, we thus have
%
\begin{eqnarray}\label{eqZ1}
\qquad Z_N(t)&=&\frac{X_N(1)+1-\lambda-M_N(1)}{N^t}+\frac{M_N(N^t)}{N^t}
\nonumber\\[-8pt]\\[-8pt]\nonumber
&&{}+\lambda-\mu\int_0^t \frac{\log(Z_N(t-v))+(t-v)\log N}{\log
(Z_N(t-v))+(t-v+1)\log N} (\log
N)N^{-v} \,\diff v.
\end{eqnarray}
Let us first show that the martingale term does not play a role for
this scaling. Indeed, for $0<a<b<1$, Doob's inequality yields, for every
$\varepsilon>0$,
\begin{eqnarray*}
\P \biggl(\sup_{a\leq s\leq b} \frac{\llvert  M_N(N^s)\rrvert  }{N^s}\geq \varepsilon
\biggr) &\leq&\P \biggl(\sup_{N^a\leq x\leq N^b} \frac
{\llvert  M_N(x)\rrvert  }{N^a}\geq
\varepsilon \biggr)
\\
&\leq&\frac{1}{\varepsilon
^2N^{2a}}\E\bigl(\bigl\llvert M_N\bigl(N^b
\bigr)\bigr\rrvert ^2\bigr)=\frac{1}{\varepsilon^2N^{2a}}\E\bigl( \langle
M_N \rangle\bigl(N^b\bigr)\bigr)
\\
&\leq& \frac{(\lambda+\mu)}{\varepsilon^2}N^{b-2a}.
\end{eqnarray*}

The last term can be made arbitrarily small when $N$ is large by
choosing $b<2a$. Since any interval $[s_0,t_0]\subset(0,{\alpha^*})$ can
be covered by a finite number of such intervals, the martingale term is
indeed negligible with probability tending to 1 as~$N$ goes to
infinity. The relation $X_N(1)\leq{\mathcal N}_\lambda([0,1])$
implies that the first term in the right-hand side of (\ref{eqZ1}) vanishes too
when divided by $N^t$.

Next, the inequality $X_N(s)\leq{\mathcal N}_\lambda([0,s])$ gives
us that
\begin{eqnarray*}
&& \int_0^t \frac{\log(Z_N(t-v))+(t-v)\log N}{\log
(Z_N(t-v))+(t-v+1)\log N} (\log
N)N^{-v} \,\diff v
\\
&&\qquad \leq Y_N(t)
\\
&&\hspace*{-2pt}\qquad \stackrel{\mathrm{def.}} {=} \int_0^t
\frac{\log
((1+{\mathcal N}_\lambda([0,N^{t-v}]))/N^{t-v})+(t-v)\log
N}{\log((1+{\mathcal N}_\lambda
([0,N^{t-v}]))/N^{t-v})+(t-v+1)\log N} (\log N)N^{-v} \,\diff v.
\end{eqnarray*}
For $0\leq v\leq t\leq b$,
\begin{eqnarray*}
&& \biggl\llvert \frac{\log((1+{\mathcal N}_\lambda
([0,N^{v}]))/N^{v})+v\log N}{\log((1+{\mathcal N}_\lambda
([0,N^{v}]))/N^{v})+(v+1)\log N}- \frac{v}{1+v}\biggr\rrvert
\\
&&\qquad\leq\frac{\log(1+S_N(b))}{\log N}\frac{1}{\log(1+S_N(b))/\log(N)+1},
\end{eqnarray*}
with
\[
S_N(b)=\sup_{0\leq v\leq b} \frac{{\mathcal N}_\lambda([0,N^{v}])}{N^{v}}.
\]
By the law of large numbers for Poisson processes, for any
$0<\varepsilon<b$, the process $({\mathcal N}_\lambda
([0,N^{s}])/N^{s},  \varepsilon\leq s\leq b)$ converges in
distribution to $(\lambda,  \varepsilon\leq s\leq b)$.
The sequence of random variables $(S_N(b))$ is therefore tight. One
gets that
\[
\sup_{a\leq t\leq b} \biggl\llvert Y_N(t)-\int
_0^t \frac{t-v}{1+t-v}(\log N)N^{-v}
\,\diff v \biggr\rrvert
\]
converges to $0$ in distribution. It is not difficult to check that the
convergence
%
\begin{equation}
\label{meur} \lim_{N\to+\infty} \int_0^t
\frac{t-v}{1+t-v}(\log N)N^{-v} \,\diff v= \frac{t}{t+1},
\end{equation}
occurs uniformly for $a\leq t \leq b$,
one gets therefore the convergence in distribution
\[
\lim_{N\to+\infty} \bigl(Y_N(t), a\leq t\leq b\bigr)=
\biggl(\frac{t}{t+1}, a\leq s\leq b \biggr).
\]
Gathering these estimates, we obtain that for any $0 <a<b<{\alpha
^*}\wedge
1$ and any $\varepsilon>0$,
\[
\lim_{N\to+\infty}\P \biggl(\lambda-\mu\frac{b}{b+1} -\varepsilon
\leq\inf_{a\leq s\leq b} Z_N(s)\leq\sup_{a\leq s\leq
b}
Z_N(s) \leq\lambda+\varepsilon \biggr)=1,
\]
note that $\lambda-\mu b/(b+1)>0$ for $b<{\alpha^*}$. Therefore, for any
$0<a_0<a\leq b\leq b_0<{\alpha^*}$, on some event ${\mathcal E}$ with
an arbitrarily high probability, the process $(\llvert  \log(Z_N(s))\rrvert, a_0\leq
s\leq b_0)$ can be bounded by some constant. Consequently, on this
event, with a similar uniform convergence argument for~(\ref{meur}),
one gets that the sequence of processes
\[
\biggl(\int_0^{t-a_0} \frac{\log(Z_N(t-v))+(t-v)\log N}{\log
(Z_N(t-v))+(t-v+1)\log N} (\log
N)N^{-v} \,\diff v, a_0\leq t\leq b_0 \biggr)
\]
converges in distribution to $({t}/{(t+1)}, a_0\leq s\leq b_0)$.
The remaining term contributing in the integral of the right-hand side
of relation~(\ref{eqZ1}) is
\begin{eqnarray*}
&& \left| \int_{t-a_0}^t
\frac{\log(Z_N(t-v))+(t-v)\log N}{\log
(Z_N(t-v))+(t-v+1)\log N} (\log N)N^{-v} \,\diff v\right|
\\
&&\qquad \leq \int
_{t-a_0}^t (\log N)N^{-v} \,\diff v
\end{eqnarray*}
and thus converges to $(0)$ as a process on $[a,b]$. The desired
convergence in distribution has thus been proved.
\end{pf*}
The following proposition shows that, if $(L_1^N(0),L_2^N(0))=(0,N)$
and $X^N(0)=0$, then the two processes $(L_1^N(t))$ and $(X^N(t))$
are close on the time scale $t\mapsto N^t$, $0<t<{\alpha^*}$. As a
consequence, it implies that the convergence result of Proposition~\ref
{asympX} is also valid for the process $(L_1^N(t))$.

\begin{proposition}\label{initprop}
If $(L^N_1(t),L^N_2(t))$ is the solution to the SDE~(\ref{SDE}) with
initial condition $(0,N)$ and ${\alpha^*}={\rho_1}/{(1-\rho_1)}$ with
$\rho_1={\lambda_1}/{\mu_1}$, then the convergence
\[
\lim_{N\to+\infty} \biggl(\frac{L^N_1(N^t)}{N^t}, 0<t<{\alpha^*} \wedge1
\biggr)= \biggl(\lambda_1- \mu_1\frac{t}{t+1},
0<t<{\alpha^*} \wedge1 \biggr)
\]
holds for the uniform topology on compact sets of $(0,{\alpha^*}\wedge1)$.
\end{proposition}

\begin{pf}
The idea of the proof is quite simple. First one shows that, on the
time scale $t\to N^t$ with $t<1$, the second component do not change
much so that its log is equivalent to $\log N$. On this time scale the
first coordinate grows linearly as long as $t<{\alpha^*}$ and for this
reason its log is equivalent to $t\log N$, so the capacity it receives
is $t/(t+1)$ which explains the result.

Since $L_2^N(0)=N$ and the number of jobs of class $2$ decreases at
rate at most $\mu_1$ and increases at rate $\lambda_1$, for any
$0<b<1$, there exist two constants $0<\eta<\gamma$ such that if
\[
{\mathcal A}_N\stackrel{\mathrm{def.}} {=} \Bigl\{\eta N\leq\inf
_{0\leq s\leq N^b}L_2^N(s)\leq\sup
_{0\leq s\leq N^b}L_2^N(s)\leq \gamma N \Bigr\},
\]
then $\P({\mathcal A}_N)$ tends to $1$ as $N$ tends to infinity. On
the set ${\mathcal A}_N$, the jump rate for departures of $L_1^N$ lies
between $\mu_1W_1^\eta(\cdot,N)$ and $\mu_1W_1^\gamma(\cdot,N)$, where
\[
W_1^\delta(x,N)=\frac{\log(1+x)}{\log(1+x)+\log(\delta N)}.
\]
Now if $(X_N^\delta(t))$ denotes the solution to equation~(\ref{eqX})
with $W_1$ replaced by $W_1^\delta$, a~straightforward coupling shows
that on the set ${\mathcal A}_N$ the relation
\[
X_N^\eta(s)\leq L_1^N(s)\leq
X_N^\gamma(s), \qquad0\leq s<N^b
\]
holds almost surely. A glance at the proof of the convergence result of
Proposition~\ref{asympX} shows that this result also holds for both
processes $(X_N^\eta(s))$ and $(X_N^\gamma(s))$, and so the
proposition is proved.
\end{pf}
The above proposition shows that if ${\alpha^*}<1$ (i.e., $\rho_1<1/2$),
then on the time scale $t\mapsto N^t$, $0<t<{\alpha^*}$, we have
\[
L_1^N\bigl(N^t\bigr)\sim \biggl(
\lambda_1 -\mu_1\frac{t}{t+1} \biggr)N^t.
\]
In particular, the process $L_1^N$ reaches the quantity $N^{{\alpha^*}
-\varepsilon}$ for any $0<\varepsilon<{\alpha^*}$. Note that, when
$t{\nearrow}{\alpha^*}$, the quantity multiplying $N^t$ vanishes, so\vspace*{1pt}
that this convergence result does not show that the value $N^{\alpha^*}$
is indeed reached. In Sections~\ref{ts2} and~\ref{ts3}, we shall
prove that the process $L_1^N$ lives in fact in a ``small''
neighborhood of~$N^{\alpha^*}$. This local equilibrium around
$N^{\alpha^*}
$ is the key phenomenon to grasp in order to understand this bandwidth
sharing policy. For now, we conclude this section by proving that for
any $0<\delta<1$, the value $\delta N^{{\alpha^*}\wedge1}$ is reached.
This is done by providing an estimation of the corresponding hitting time.

\begin{proposition}\label{Job}
With the same notation and assumptions as in Proposition~\ref
{initprop}, and if $H_a$ denotes the hitting time of $a>0$ by $L_1^N$,
\[
H_a=\inf\bigl\{t>0\dvtx  L_1^N(t)\geq a\bigr\},
\]
then:
\begin{longlist}[(a)]
\item[(a)] if ${\alpha^*}<1$, for any $0<\delta<1$ there exists a constant
$C>0$ such that for any $N\geq1$,
\[
\E (H_{\delta N^{{\alpha^*}}} )\leq\frac{C \delta}{\log
(1/\delta)} N^{{\alpha^*}}\log N,
\]
\item[(b)] if ${\alpha^*}>1$, for $\delta>0$ sufficiently small we have
\[
\limsup_{N\to+\infty} \frac{1}{N} \E (H_{\delta N} )\leq
\frac{\delta}{\lambda_1-\mu_1/2}.
\]
\end{longlist}
\end{proposition}

\begin{pf}
As we did for Proposition~\ref{initprop}, let us prove these two
inequalities for the process $(X_N(t))$. We use the simplified notation
of Proposition~\ref{asympX}.

Let us first assume that ${\alpha^*}<1$. For $x>0$, the elementary relation
%
\begin{eqnarray}\label{eqelem}
&&\frac{\log(1+x)}{\log(1+x)+\log N}-\frac{{\alpha^*}}{{\alpha
^*}+1}
\nonumber\\[-8pt]\\[-8pt]\nonumber
&&\qquad = \frac{\log((1+x)/N^{\alpha^*})}{({\alpha^*}+1)(\log N)(1+{\alpha
^*}+\log
((1+x)/N^{\alpha^*})/(\log N))},
\end{eqnarray}
together with the identity
\[
\lambda=\mu \frac{{\alpha^*}}{1+{\alpha^*}}
\]
and equation~(\ref{eqXint}) written at time $N^{\alpha^*}(\log N)  t$
give the representation
%
\begin{eqnarray}
\label{aux1} Z_N(t)&\stackrel{\mathrm{def.}} {=}&\frac{X_N(N^{\alpha^*}(\log N)
t)}{N^{\alpha^*}}\nonumber
\\
&=& \frac{M_N(N^{\alpha^*}(\log N) t)}{N^{\alpha^*}}
\\
&&{}-\frac{\mu
}{(1+{\alpha^*})}\int_0^t
\frac{\log(N^{-{\alpha
^*}}+Z_N(u))}{{\alpha^*}
+1+\log(N^{-{\alpha^*}}+Z_N(u))/(\log N)} \,\diff u.\nonumber
\end{eqnarray}
Let
\[
\tau_N=\inf \bigl\{s>0\dvtx  X_N \bigl(N^{\alpha^*}(\log
N) s \bigr)\geq \delta N^{\alpha^*} \bigr\}.
\]
From Doob's optional stopping time theorem and the fact that $Z_N(\tau
_N\wedge t) \leq N^{-{\alpha^*}} \lceil\delta N^{\alpha^*}\rceil$, we
obtain the inequality
\[
\frac{\lceil\delta N^{\alpha^*}\rceil}{N^{\alpha^*}}+\frac{\mu
}{(1+{\alpha^*})}\frac{\log(\delta+2N^{-{\alpha^*}})}{{\alpha
^*}+1+\log
(\delta+2N^{-{\alpha^*}})/(\log N)} \E(\tau_N)
\geq0.
\]
Since $H_{\delta N^{{\alpha^*}}}\leq N^{\alpha^*}(\log N)\tau_N$,
item (a) is proved.

Assume now that ${\alpha^*}>1$, that is, that $\rho>1/2$. Equation~(\ref
{eqXint}) written at time $ N  t$ gives the relation
\[
X_N(Nt)= M_N(Nt)+\lambda Nt-\mu\int_0^{Nt}
\frac{\log
(1+X_N(s))}{\log(1+X_N(s))+\log N} \,\diff s,
\]
from which we can write that
\begin{eqnarray*}
\lceil\delta N\rceil&\geq&\E\bigl(X_N\bigl(H_{\delta N}\wedge(Nt)
\bigr)\bigr)
\\
&\geq&\E \bigl(H_{\delta N}\wedge(Nt)\bigr) \biggl(\lambda-\mu
\frac{\log(1+\lceil
\delta N\rceil)}{\log(1+\lceil\delta N\rceil)+\log N} \biggr).
\end{eqnarray*}
Letting $t$ go to infinity, the monotone convergence theorem gives us that
%
\begin{equation}
\label{upperbound} \limsup_{N\to+\infty}\frac{1}{N}
\E(H_{\delta N})\leq\frac
{\delta}{\lambda-\mu/2}.
\end{equation}
If we now choose $\delta$ sufficiently small so that, with high
probability, the component $L_2^N(t)$ is still of the order of $N$ at
time $\delta N/(\lambda-\mu/2)$, a straightforward coupling between
$L_1^N$ and some $X_N^\eta$ (recall the notation $X_N^{\eta}$ from
the proof of Proposition~\ref{initprop}) ensures that (\ref{upperbound}) holds as well for the process $(L_1^N(t))$. This completes the
proof of Proposition~\ref{Job}.
\end{pf}

\section{A local equilibrium}\label{ts2}
This section is essentially devoted to the behavior of our
two-dimensional process on the time scale $t\mapsto N^{\alpha^*}(\log
N)t$, when the initial state is $(L_1^N(0),L_2^N(0))=(\delta N^{{\alpha^*}
},N)$ and $\rho_1<1/2$. The following result, Proposition~\ref
{Fluid}, shows that on this time scale the sample paths of $(L_1^N(t))$
have values of the order of $x N^{\alpha^*}$, where $0<x<1$. When the
initial value of $(L_1^N(t))$ is $N^{\alpha^*}$, we shall prove in
Theorem~\ref{theostab} that the process is stabilized around
$N^{\alpha^*}$. At first sight, the interest of Proposition~\ref{Fluid}
and of the associated central limit theorem (Proposition~\ref{tclgen}
below) may seem marginal. This is not true at all since, as we shall
see in Section~\ref{ts3}, the process is also of the order of $\gamma
(t)N^{\alpha^*}$ on the fluid time scale $t\mapsto Nt$. Furthermore, the
main difficulty in the key technical result of Theorem~\ref{lem1} is
precisely connected to the interaction of these two time scales
$t\mapsto N^{\alpha^*}(\log N)t$ and $t\mapsto Nt$.

Recall the notation ${\alpha^*}=\rho_1/(1-\rho_1)$.

\begin{proposition}\label{Fluid}
If $\rho_1<1/2$ and $(L^N_1(t),L^N_2(t))$ is the solution to the
SDE~(\ref{SDE}) with initial conditions $L_2^N(0)=N$ and $L_1^N(0)\sim
\delta N^{\alpha^*}$ for some $\delta\in(0,1]$, then the sequence of
stochastic processes
\[
\biggl(\frac{L_1^N(N^{\alpha^*}(\log N) t)}{N^{\alpha^*}} \biggr)
\]
converges in distribution to $(h(t))$ defined by
%
\begin{equation}
\label{eqh} \cases{ h\equiv1, &\quad if $\delta=1$,
\vspace*{3pt}\cr
\displaystyle\int
_{\delta}^{h(t)}\frac{1}{\log(u)} \,\diff u=-
\frac{\mu
t}{(1+{\alpha^*})^2}, &\quad if $\delta\neq1$.}
\end{equation}
\end{proposition}

\begin{pf}
As we did for Proposition~\ref{initprop}, the convergence is proved
for the process $(X_N(t))$. The result for $(L_1^N(t))$ follows from a
similar coupling argument.

Let us first show that if $X_N(0)=\lfloor\delta N^{\alpha^*}\rfloor$,
then $X_N(t)/N^{\alpha^*}$ remains within $[\delta/3,3]$ on a given time
interval $[0,N^{\alpha^*}(\log N)T]$ with probability tending to~$1$.
Indeed, writing $H_a$ for the hitting time of $a$ by $X_N(N^{\alpha^*}
(\log N)\cdot)$, the strong Markov property of $(X_N(t))$ gives us that
\begin{eqnarray}\label{strongMarkov}
\quad \P (H_{3N^{\alpha^*}}\leq T )
&=& \E \bigl(\mathbf{1}_{\{
H_{2N^{\alpha^*}}<T\}}
\P_{\lceil2N^{\alpha^*}\rceil} (H_{3N^{\alpha^*}
}\leq T-H_{2N^{\alpha^*}}\mid
H_{2N^{\alpha^*}} \leq T ) \bigr)
\nonumber\\[-8pt]\\[-8pt]\nonumber
&\leq&  \P (H_{2N^{\alpha^*}}<T) \P_{\lceil2N^{\alpha^*}\rceil
} (H_{3N^{\alpha^*}}\leq T ) .\nonumber
\end{eqnarray}
Now, due to the monotonicity properties of the service rate of
$(X_N(t))$, if $X_N(0)=2N^{\alpha^*}$ then we can couple $(X_N(t))$ with
the process $(2N^{\alpha^*}+ R_N(t))$ defined by:
\begin{itemize}
\item[--] $R_N \rightarrow R_N+1$ at rate $\lambda$,
\item[--] $R_N \rightarrow R_N-1$ at rate
\[
\mu \frac{\log(2N^{\alpha^*})}{\log(2N^{\alpha^*})+\log N} = \mu \frac{{\alpha^*}}{{\alpha^*}+1} + \frac{C}{\log N},
\]
for some $C>0$,
\item[--] $R_N(0)=0$ and $(R_N(t))$ reflects at $0$,
\end{itemize}
in such a way that $X_N(t)\leq2N^{\alpha^*}+ R_N(t)$ for every $t\geq0$.
Hence, there remains to prove that $(R_N(t))$ does not reach $N^{\alpha^*}
$ in less than $N^{\alpha^*}(\log N)T$ units of time. But by Kingman's
inequality (see Kingman~\cite{Kingman} or relation~(3.3) of
Theorem~3.5 in Robert~\cite{Robert}) and the fact that $\lambda=\mu
{\alpha^*}/(1+{\alpha^*})$, if $\theta_N$ stands for the time span
of the
first excursion of $(R_N(t))$ away from $0$ we have
\[
\P_1 \Bigl(\sup_{s\in[0,\theta_N]}R_N(s)>N^{\alpha^*}
\Bigr) \leq \exp \biggl(-\frac{C'N^{\alpha^*}}{\log N} \biggr) %
\]
for some explicit $C'>0$. Since the $i$th pair of consecutive
excursions is separated by the amount $E_i$, an exponentially
distributed random variable with parameter $\lambda$, and, since these
exponential times are independent of each other, a Chernoff bound on
Poisson random variables gives the relation
\begin{eqnarray*}
&& \P \bigl(\mbox{more than }2\lambda N^{\alpha^*}(\log N)T \mbox{ excursions
before time }N^{\alpha^*}(\log N)T \bigr)
\\
&&\qquad  \leq\P \Biggl(\sum_{i=1}^{2\lambda N^{\alpha^*}(\log N)T}E_i
\leq N^{\alpha^*}(\log N)T \Biggr) \leq e^{-C''N^{\alpha^*}(\log N)T}
\end{eqnarray*}
for some $C''>0$. Coming back to (\ref{strongMarkov}), we obtain that
\[
\P (H_{3N^{\alpha^*}}\leq T )\leq e^{-C''N^{\alpha^*}(\log
N)T} + 2\lambda N^{\alpha^*}(
\log N)T \exp \biggl\{-\frac{C'N^{\alpha
^*}}{\log
N} \biggr\}, %
\]
which tends to $0$ as $N$ tends to infinity. Finally, since the
infinitesimal drift of $X_N(t)$ is positive when $X_N(t)< N^{{\alpha^*}}$,
the same method can be used to show that $(X_N(t))$ remains above
$(\delta/3)N^{\alpha^*}$ on the time interval $[0,N^{\alpha^*}(\log N)T]$,
with overwhelming probability.

Now let $(w_{f}(\xi))$ denote the modulus of continuity of the
function $(f(t))$ on $[0,T]$. That is,
\[
w_{f}(\xi)=\mathop{\sup_{s,t\leq T}}_{\llvert  s-t\rrvert  \leq\xi}
\bigl\llvert f(s)-f(t)\bigr\rrvert.
\]
Let us also write again
\[
Z_N(t)\stackrel{\mathrm{def.}} {=} \frac{X_N(N^{\alpha^*}(\log
N)t)}{N^{\alpha^*}}. %
\]
Using the bounds on $(X_N(t))$ we have just obtained, we can deduce the
existence of a constant $A$ independent of $N$ such that, with
probability tending to $1$, we have, for any $s<t\in[0,T]$,
\[
\int_s^t \biggl\llvert \frac{\log(N^{-{\alpha^*}}+Z_N(u))}{{\alpha
^*}+1+\log
(N^{-{\alpha^*}}+Z_N(u))/(\log N)}
\biggr\rrvert \,\diff u \leq A(t-s). %
\]
Together with relation~(\ref{aux1}) and the fact that its martingale
term vanishes as $N$ gets large, we can conclude that for any
$\varepsilon>0$ and $\eta>0$, there exists $\xi>0$ such that
\[
\P\bigl(w_{Z_N}(\xi)>\eta\bigr)\leq\varepsilon
\]
holds for all $N$. By the tightness criterion of the modulus of
continuity (see Theorem~8.3 in Billingsley~\cite{Billingsley}), the
sequence of processes $(Z_N(t))$ is thus tight and any limiting point
$h$ satisfies the relation
\[
h(t)=\delta-\frac{\mu}{(1+{\alpha^*})^2}\int_0^t \log
\bigl(h(u)\bigr) \,\diff u.
\]
It is easily seen that there is a unique solution to this integral
equation and that its solution can be expressed as the solution to the
fixed point equation~(\ref{eqh}).
Proposition~\ref{Fluid} is proved.
\end{pf}
The above proposition can be seen as a kind of law of large numbers, with
\[
L_1^N\bigl(N^{\alpha^*}\log(N) t\bigr)\sim h(t)
N^{\alpha^*}\qquad\mbox{as }N\rightarrow\infty.
\]
It is thus natural to expect for the corresponding central limit
theorem that
\[
\biggl(\frac{L_1^N(N^{\alpha^*}\log(N) t)- h(t) N^{\alpha^*}}{\sqrt
{N^{\alpha^*}\log(N)}} \biggr)
\]
should converge in distribution to some Gaussian process $(R(t))$.
However, the following proposition shows that such a convergence cannot
hold: the centering term $h(t)$ has to be replaced by a deterministic
term $h_N(t)$ which depends on $N$. As we shall see, $h_N$ is such that
the following expansion holds:
\[
h_N(t)= h(t) -\frac{\mu_1}{({\alpha^*}+1)^3\log N}\int_0^t
\log \bigl(h(u)\bigr)^{2} \,\diff u +o \biggl(\frac{1}{\log N }
\biggr).
\]
In particular, $(h_N(t))$ converges ``slowly'' to $(h(t))$ at the rate
$1/\log N$ instead of a rate $1/N^{{\alpha^*}/2+\varepsilon}$ for which
a ``classic'' centering procedure would be enough. In the end, this
second limit theorem gives the representation
\[
L_1\bigl(N^{\alpha^*}\log(N) t\bigr)= h_N(t)
N^{\alpha^*}+ O \bigl(\sqrt {N^{{\alpha^*}
}\log N} \bigr).
\]

\begin{proposition}[(Central limit theorem)]\label{tclgen}
If $\rho_1<1/2$ and $(L^N_1(t),L^N_2(t))$ is the solution to the
SDE~(\ref{SDE}) with the initial conditions $L_2^N(0)=N$ and
$L_1^N(0)$ being such that
\[
\lim_{N\to+\infty} \frac{L_1^N(0)-\delta N^{\alpha^*}}{\sqrt
{N^{\alpha^*}\log N}}=y,
\]
for some $0<\delta\leq1$ and $y\in\R$. Then we have
\[
\lim_{N\to+\infty} \biggl(\frac{L_1^N(N^{\alpha^*}(\log N) t)-
h_N(t)N^{\alpha^*}}{\sqrt{N^{\alpha^*}\log N}} \biggr)= \bigl(R(t)
\bigr),
\]
for the convergence in distribution, where $(h_N(t))$ is the solution
to the ordinary differential equation
\[
\dot{h}_N(t)= -\frac{\mu_1}{(1+{\alpha^*})}\frac{\log
h_N(t)}{({\alpha^*}+1+\log(h_N(t))/\log N)},
\]
with $h_N(0)=\delta$ and $(R(t))$ is the solution to the following SDE:
%
\begin{equation}
\label{eqR} \diff R(t)=\sqrt{2\lambda_1} \,\diff B(t)-
\frac{\mu_1}{(1+{\alpha^*}
)^2} \frac{R(t)}{h(t)} \,\diff t,
\end{equation}
with $R(0)=y$, $(B(t))$ denoting standard Brownian motion on $\R$ and
$(h(t))$ being defined in (\ref{eqh}).
\end{proposition}

\begin{pf}
From relations~(\ref{croc}) and (\ref{eqelem}), the increasing
process of the martingale
\[
\bigl(\overline{M}_N(t) \bigr)\stackrel{\mathrm{def.}} {=} \biggl(
\frac{M_N(N^{\alpha^*}(\log N) t)}{\sqrt{N^{\alpha^*}\log
N}} \biggr)
\]
is given by
\[
\biggl(2\lambda t +\frac{\mu}{(1+{\alpha^*})\log N}\int_0^t
\frac
{\log
(N^{-{\alpha^*}}+Z_N(u))}{{\alpha^*}+1+\log(N^{-{\alpha
^*}}+Z_N(u))/(\log N)} \,\diff u \biggr).
\]
By Proposition~\ref{Fluid}, this quantity converges in distribution to
$(2\lambda t)$. Hence, using the martingale criterion for convergence
to a Brownian motion (see Theorem~1.4, page~339 of Ethier and
Kurtz~\cite{Ethier01}) we can conclude that $ (\overline
{M}_N(t) )$ converges in distribution to Brownian motion run at
speed $2\lambda$.

With the same notation as in the proof of the previous proposition, the
relation satisfied by $(Z_N(t))$ yields
%
\begin{eqnarray}\label{eqZ2}
\qquad \overline{Z}_N(t)&\stackrel{\mathrm{def.}} {=}&\sqrt{
\frac{N^{\alpha^*}
}{\log N}} \bigl(Z_N(t)-h_N(t)
\bigr)\nonumber
\\
&=&\overline {Z}_N(0)+\overline{M}_N(t)
\\
&&{}-\frac{\mu}{(1+{\alpha^*})^2} \int_0^t \sqrt{
\frac{N^{\alpha^*}}{\log N}} \bigl(d_N \bigl(N^{-{\alpha^*}
}+Z_N(u)
\bigr)-d_N\bigl(h_N(u)\bigr) \bigr) \,\diff u,\nonumber
\end{eqnarray}
where
\[
d_N(x)=\frac{\log x}{1+(\log x)/(({\alpha^*}+1)\log N)}.
\]
Let $T>0$ and $\varepsilon>0$, and let $A_N$ be the event
\[
A_N= \biggl\{\inf_{0\leq s\leq T} Z_N(s)\geq
\frac{\delta}{2} \biggr\}.
\]
By Proposition~\ref{Fluid} and the fact that $(h(t))$ is nondecreasing
with $h(0)=\delta$, there exists $N_0$ such that for any $N\geq N_0$,
$\P(A_N^c)\leq\varepsilon$.

The properties of the derivate $d_N'$ of $d_N$ are used to analyze the
limit of $(\overline{Z}_N(t))$. Observe that
%
\begin{equation}
\label{derivee} d_N'(x) = \biggl\{x \biggl(1+
\frac{\log x}{({\alpha^*}+ 1)\log N} \biggr) \biggr\}^{-1}.
\end{equation}
On the event $A_N$, we obtain from relation~(\ref{eqZ2}) and the fact
that for $N$ large enough,
\[
\sup_{[\delta/2,\infty)} d_N' \leq
\frac{4}{\delta} %
\]
that for any $t\leq T$,
\begin{eqnarray*}
&& \bigl\llvert \overline{Z}_N(t)\bigr\rrvert \leq\bigl\llvert
\overline{Z}_N(0)\bigr\rrvert +\bigl\llvert \overline{M}_N(t)
\bigr\rrvert
+\frac{4 \mu}{\delta(1+{\alpha^*})^2} \int_0^t \bigl\llvert
\overline{Z}_N(u)\bigr\rrvert \,\diff u.
\end{eqnarray*}
As a consequence, Gronwall's lemma gives us that on the set $A_N$
\[
\sup_{0\leq t\leq T}\bigl\llvert \overline{Z}_N(t)\bigr
\rrvert \leq \Bigl(\bigl\llvert \overline{Z}_N(0)\bigr\rrvert +\sup
_{0\leq t\leq T} \bigl\llvert \overline {M}_N(t)\bigr\rrvert
\Bigr)e^{CT},
\]
where $C=4\mu/(2(1+{\alpha^*})^2)$. Using the fact that the sequence of
processes $(M_N(t))$ is tight for the topology of uniform convergence,
together with our assumption on $X_N(0)$, we obtain that there exist
$N_1$ and $K>0$ such that for any $N\geq N_1$
\[
\P\bigl(B_{N,K}^c\bigr)\leq\varepsilon\qquad\mbox{where }
B_{N,K}= \biggl\{ \inf_{0\leq s\leq T} Z_N(s)\geq
\frac{\delta}{2}, \sup_{0\leq t\leq
T}\bigl\llvert
\overline{Z}_N(t)\bigr\rrvert \leq K \biggr\}.
\]
Next, recall the notation $w_f$ for the modulus of continuity of the
function $f$. On the event $B_{N,K}$, relation~(\ref{eqZ2}) gives us
that for any $\xi>0$,
\[
w_{\overline{Z}_N}(\xi)\leq w_{\overline{M}_N}(\xi)+ C \mathop{\sup
_{s,t\leq T}}_{\llvert  s-t\rrvert  \leq\xi}\int_s^t
\bigl\llvert \overline {Z}_N(u)\bigr\rrvert \,\diff u\leq
w_{\overline{M}_N}(\xi)+CK\xi.
\]
The sequence of processes $(\overline{M}_N(t))$ being tight, this
relation and the fact that $\P(B_{N,K})>1-\varepsilon$ show that for
any $\eta>0$, there exists $\xi_0$ such that for every $\xi\leq\xi
_0$ and $N\geq N_0$,
\[
\P\bigl(w_{Z_N}(\xi)\geq\eta\bigr)\leq3\varepsilon.
\]
Since we can apply this reasoning to any $\varepsilon>0$, here again
the tightness criterion of the modulus of continuity enables us to
conclude that the sequence of processes $(\overline{Z}_N(t))$ is tight.

Let $(R(t))$ be one of the limiting points. By using Skorohod's
representation theorem (see Ethier and Kurtz~\cite{Ethier01}) up to a
change of probability space, we can assume that the convergence to
$(R(t))$ holds almost surely on $[0,T]$ for the uniform norm. But on
the set $B_{N,K}$, (\ref{derivee}) and Lebesgue's differentiation
theorem (see Rudin~\cite{Rudin}, e.g.) guarantees that the
integral term on the right-hand side of equation~(\ref{eqZ2})
converges almost surely to
\[
\int_0^t \frac{R(u)}{h(u)} \,\diff u.
\]
Consequently $(R(t))$ satisfies the SDE~(\ref{eqR}) with $R(0)=y$, and
the uniqueness of such a solution gives us the convergence in
distribution we were seeking.
\end{pf}
A direct consequence of this result is that, starting from $N^{\alpha^*}
$, the process $(L_1^N(t))$ behaves like an Ornstein--Uhlenbeck process
around $N^{\alpha^*}$.

\begin{theorem}[{[A stable regime for $(L_1^N(t))$]}]\label{theostab}
If $\rho_1<1/2$, $L_2^N(0)=N$ and $L_1^N(0)$ is such that, for some
$y\in\R$,
\[
\lim_{N\to+\infty}\frac{L_1^N(0)- N^{{\alpha^*}}}{\sqrt
{N^{{\alpha^*}
}\log N}}=y,
\]
then the sequence of processes
\[
\biggl(\frac{L_1^N (N^{\alpha^*}(\log N) t )-N^{\alpha
^*}}{\sqrt
{N^{{\alpha^*}}\log N}} \biggr)
\]
converges in distribution to an Ornstein--Uhlenbeck process $(Z(t))$,
that is, the solution to
the SDE
%
\begin{equation}
\label{OU} \diff Z(t)=\sqrt{2\lambda_1} \,\diff B(t)-
\frac{\mu_1}{({\alpha^*}
+1)^2} Z(t)\,\diff t, \qquad Z(0)=y,
\end{equation}
where $(B(t))$ denotes standard Brownian motion.
\end{theorem}

To complete this section on the time scale $t\to N^{{\alpha^*}}\log N$,
the following proposition shows that, if $\rho_1<1/2$ and
$L_2^N(0)=N$, then the process $(L_1^N(t))$ always\vspace*{1pt} reaches the stable
regime around $N^{\alpha^*}$ described in the previous theorem. In
particular, if $L_1^N(0)=0$, the hitting time of $\lfloor N^{\alpha^*}
\rfloor$ is smaller than $N^\beta$ for any $\beta>{\alpha^*}$.

\begin{proposition}\label{tajine}
Let
\[
T_N\stackrel{\mathit{def.}} {=}\inf \bigl\{s>0\dvtx
L_1^N(s)\in N^{\alpha^*}+ \bigl[-
\sqrt{N^{\alpha^*}\log N}, \sqrt{N^{\alpha
^*}\log N} \bigr] \bigr\}.
\]
If $\rho_1<1/2$, $L_2^N(0)=N$ and $L_1^N(0)\leq N^\beta$ for some
$\beta\in({\alpha^*},1)$, then
\[
\lim_{N\to+\infty}\P \biggl(\frac{T_N}{N^{\beta}(\log N)^2}\leq 1 \biggr)=1.
\]
\end{proposition}

\begin{pf}
As before, the result is proved for the process $(X_N(t))$ instead of
$(L_1^N(t))$.
Suppose that $X_N(0)=\lfloor N^\beta\rfloor$. The SDE~(\ref{eqXint})
and relation~(\ref{eqelem}) show that for any stopping time $\tau$,
one has
%
\begin{eqnarray}\label{Dyn}
\E\bigl(X_N(t\wedge\tau)\bigr)&=& X_N(0)\nonumber
\\
&&{}-\frac{\mu}{({\alpha^*}+1)\log N}
\\
&&\quad{}\times\E \biggl(\int_0^{t\wedge\tau}
\frac{\log((1+X_N(u))/N^{\alpha^*})}{1+{\alpha^*}+\log
((1+X_N(u))/N^{\alpha^*}
)/\log N} \,\diff u \biggr).\nonumber
\end{eqnarray}
Defining again
\[
H_x\stackrel{\mathrm{def.}} {=}\inf\bigl\{s>0\dvtx  L_1(s)=
\lfloor x\rfloor\bigr\},
\]
and setting $x_0^N=2\lceil N^{{\alpha^*}}\rceil$, the above relation gives
\[
0\leq X_N(0)-\frac{\mu}{({\alpha^*}+1)\log N}\frac{\log
2}{1+{\alpha^*}
+(\log2)/\log N} \E (t\wedge
H_{x_0^N} ).
\]
Consequently, letting $t$ and then $N$ go to infinity yields
\[
\limsup_{N\to+\infty} \frac{\E(H_{x_0^N})}{N^\beta\log N} <+\infty.
\]
We can therefore assume that $X_N(0)=2\lceil N^{{\alpha^*}}\rceil$.
Setting this time $x^N_1=\lfloor N^{\alpha^*}\rfloor+\lfloor
N^{{\alpha^*}
-\varepsilon}\rfloor$, where $\varepsilon>0$ is such that ${\alpha^*}
+\varepsilon<\beta$, the same argument gives us that
\[
x_1^N\leq X_N(0) -\frac{\mu}{({\alpha^*}+1)\log N}
\frac{\log
(1+1/N^\varepsilon)}{1+{\alpha^*}+\log(1+1/N^\varepsilon)/\log N} \E (H_{x_1^N} ),
\]
and thus
%
\begin{equation}
\label{ineq1} \limsup_{N\to+\infty} \frac{\E(H_{x_1^N})}{N^{{\alpha
^*}+\varepsilon
}\log N} <+\infty.
\end{equation}
Similarly, if $x_2^N=\lfloor N^{\alpha^*}\rfloor+\lfloor N^{{\alpha^*}
-2\varepsilon}\rfloor$ and if we choose $X_N(0)=x_1^N$, the above
equation gives for $\tau=H_{x_2^N}$
\[
\bigl\lfloor N^{{\alpha^*}-2\varepsilon} \bigr\rfloor\leq\bigl\lfloor N^{{\alpha^*}
-\varepsilon}\bigr
\rfloor-\frac{\mu}{({\alpha^*}+1)\log N}\frac{\log
(1+1/N^{2\varepsilon})}{1+{\alpha^*}+\log(1+1/N^{2\varepsilon
})/\log
N} \E (H_{x_2^N} ),
\]
so that relation~(\ref{ineq1}) also holds for $H_{x_2^N}$. Setting
$x_i^N=N^{\alpha^*}+N^{{\alpha^*}-i\varepsilon}$, we can proceed by
induction until the smallest integer $i^*$ such that $i^*\varepsilon
>{\alpha^*}/2$. Finally, we obtain that as $N\rightarrow\infty$,
\[
\frac{\E(T_N)}{N^\beta(\log N)^2}=\frac{\E(H_{x_0^N})}{N^\beta
(\log N)^2}+ \sum_{i=1}^{i^*}
\frac{\E
(H_{x_i^N}-H_{x_{i-1}^N})}{N^\beta(\log N)^2} \rightarrow0. %
\]
We conclude by using the Markov inequality.

Up to now we have been dealing with the case $X_N(0)>N^{\alpha^*}$. There
remains to consider the case $X_N(0)< N^{\alpha^*}$. First,
Proposition~\ref{Job} shows that we can assume directly that
$X_N(0)=\lfloor xN^{\alpha^*}\rfloor$ for some $x\in(0,1)$. We can then
proceed as before by estimating $H_{N^{\alpha^*}-N^{{\alpha
^*}-\varepsilon
}}$ for $\varepsilon$ sufficiently small and by decreasing the
exponent by~$\varepsilon$ at each step until it falls below ${\alpha^*}
/2$. This completes the proof of Proposition~\ref{tajine}.
\end{pf}

\begin{rem*}
Proposition~\ref{tajine} completes\vspace*{1pt} the results of Section~\ref{ts1}.
Indeed, it shows in particular that if $L_1^N(0)=0$ and $L_2^N(0)=0$,
then the average hitting time $\E(T_N)$ of the neighborhood of
$N^{\alpha^*}$ is, up to a constant, upper bounded by $N^\beta$ for any
$\beta>{\alpha^*}$. With the same arguments as in the previous proof, it
can be shown in fact that there exists a constant $C>0$, such that
%
\begin{equation}
\label{Goisot} \E_{(0,N)}(T_N)\leq C_1
\frac{N^{\alpha^*}\log(N )^2\log\log
(N)}{\log
\log\log(N)}.
\end{equation}
\end{rem*}

\section{The fluid time scale}\label{ts3}
Recall that the fluid scaling of $(L(t))=(L_1(t),\break  L_2(t))$ consists in
speeding up the time scale of the Markov process proportionally to the
norm of its initial state and by scaling the state variable by the same
quantity. Hence, if $\llVert  L^N(0)\rrVert  =\max(L_1^N(0),L_2^N(0))=N$, we are
interested in the process
\[
\bigl(\overline{L}_N(t) \bigr)\stackrel{\mathrm{def.}} {=}
\frac
{1}{N} \bigl(L(Nt) \bigr).
\]
Without loss of generality, it can be assumed that
\[
\lim_{N\to+\infty} \overline{L}_N(0)=\lim
_{N\to+\infty} \biggl(\frac{L_1(0)}{N},\frac{L_2(0)}{N}
\biggr)=(x,1-x),
\]
for some $0\leq x\leq1$. See, for example, Bramson~\cite{Bramson} and
Robert~\cite{Robert}.

The initial fluid state considered up to now in Propositions~\ref
{Fluid} and~\ref{tajine} corresponds to the case $x=0$,
\[
\lim_{N\to+\infty} \overline{L}_N(0)=(0,1).
\]
It has been shown in Proposition~\ref{tajine} of Section~\ref{ts2}
that in this setting the hitting time of $N^{\alpha^*}$ by $(L_1(t))$ is
negligible compared to $N$. Consequently, on the fluid time scale,
$L_1(t)$ is immediately of the order of $L_2^{\alpha^*}$.

The following proposition completes this result. It shows that for any
initial fluid state, then the first time $L_1$ is close to $L_2^{\alpha^*}
$ is of the order of $N$ and, therefore, that this event occurs on the
fluid time scale.

\begin{proposition}\label{Jac}
Suppose that $\rho_1<1/2$, $\rho_2>1/2$, and that
$(L^N_1(t),\break L^N_2(t))$ is the solution to the SDE~(\ref{SDE}) with
initial conditions $L_2^N(0)=N$ and $L_1^N(0)$ such that
\[
\lim_{N\to+\infty}{L_1^N(0)}/{N}=x \in\R_+.
\]
Let $T_N$ be defined by
\begin{eqnarray*}
T_N &=&\inf \Bigl\{ s>0\dvtx  L_1^N(s)\in
L_2^N(s)^{\alpha^*}
\nonumber\\[-8pt]\\[-8pt]\nonumber
&&\hspace*{15pt}{} + \Bigl[-\sqrt
{L_2^N(s)^{\alpha^*}\log L_2^N(s)},
\sqrt {L_2^N(s)^{\alpha^*}\log
L_2^N(s)} \Bigr] \Bigr\}.
\end{eqnarray*}
Then, for the convergence in distribution we have
\[
\lim_{N\to+\infty} \frac{T_N}{N}=t_0(x)\stackrel{
\mathit{def.}} {=}\frac{2x}{\mu_1-2\lambda_1}.
\]
\end{proposition}

\begin{pf}
To start with, note that for both $i\in\{1,2\}$ and all $t\geq0$,
$L_i(t)\leq L_i(0)+{\mathcal N}_{\lambda_i}([0,t])$. Hence, by the
law of large numbers for Poisson processes, for every $\eta>0$ and
$K>0$ there exists $N_0\in\N$, such that with probability greater
than $1-\eta$, the relations
\[
L_1(Nt)\leq(1 +2\lambda_1K) N \quad\mbox{and}\quad
L_2(Nt)\leq(1 + 2\lambda_2K) N
\]
hold for all $t\leq K$ and $N\geq N_0$.

Let us now define $\tau_0^N = \inf\{s\geq0\dvtx  L_1^N(s)\leq N/(\log
N)^2\}$. Of course, this time is $0$ when $L_1^N(0)\leq N/(\log N)^2$.
By the remark made in the previous paragraph, with probability at least
$1-\eta$ we have for every $s\leq\tau_0^N\wedge(KN)$
%
\begin{equation}
\label{ineq2} \frac{\log N-2 \log\log N}{\log(1+2\lambda_2K)-2\log\log N +2\log
N}\leq\frac{\log L_1^N(s)}{\log L_1^N(s)+\log L_2^N(s)}
\end{equation}
and
\[
\frac{\log L_2^N(s)}{\log L_1^N(s)+\log L_2^N(s)} \leq\frac{\log
(1+2\lambda_2K)+\log N}{\log(1+2\lambda_2K)-2 \log\log N +2\log N}.
\]
Hence, on this time interval the service capacity offered to class $1$
jobs (resp., class~$2$ jobs) is at least (resp., at most) $1/2$ in the
limit. Since $\rho_2>1/2$, this implies that the process $(L_2^N(s))$
is increasing on the time interval $[0,\tau_0^N\wedge(KN)]$. That is,
for $N$ sufficiently large we have with probability greater than
$1-\eta$
\[
\inf_{s\leq\tau_0^N\wedge(KN)} L_2^N(s) \geq N- \log N.
\]
Consequently, relations~(\ref{ineq2}) can be completed by the inequality
\[
\frac{\log L_1^N(s)}{\log L_1^N(s)+\log L_2^N(s)} \leq\frac{\log(x+2\lambda_1K)+\log N}{\log(x+2\lambda_1K)+\log
(1-(\log N)/N)+2\log N}.
\]
This shows that, with high probability, as $N$ gets large the two
queues receive the capacity $1/2$. As a straightforward consequence, we
have the convergence in distribution
\[
\lim_{N\rightarrow\infty} \frac{\tau_0^N}{N}= \frac{2x}{\mu
_1-2\lambda_1}=t_0(x).
\]

Next, let us suppose that $L_1^N(0)\leq\lfloor N/(\log N)^2\rfloor$
and let us define
\[
\tau_1^N=\inf \bigl\{s>0\dvtx  L_1^N(s)
\leq2 N^{\alpha^*} \bigr\}. %
\]
Since $L_2^N(Nt)$ grows at linear rate (recall that $\rho_2>1/2$), we
can again compare $L_1^N$ to $(X_N(t))$ and conclude from
relation~(\ref{Dyn}) applied to the stopping time $\tau_1^N$ that
\[
\frac{\mu_1}{({\alpha^*}+1)\log N}\frac{\log(2)}{1+{\alpha
^*}+\log
(2)/\log N} \E \bigl(\tau_1^N
\bigr) \leq L_1^N(0).
\]
Consequently,
\[
\lim_{N\to+\infty} \E \bigl(\tau_1^N
\bigr)/N=0.
\]
We can thus assume that $L_1^N(0)=2\lfloor N^{\alpha^*}\rfloor$, and
Proposition~\ref{tajine} shows that in this case,
\[
\lim_{N\to\infty} \E(T_N)/N = 0. %
\]
Coming back to the initial question (with an arbitrary $x$) and using
the strong Markov\vspace*{1pt} property of $L^N$ combined with the last two limits,
we obtain that $T_N=\tau_0^N+(\tau_1^N-\tau_0^N)+(T_N-\tau_1^N)$, where
\[
\frac{\tau_0^N}{N}\stackrel{(d)} {\rightarrow} t_0(x),\qquad
\frac
{\tau_1^N-\tau_0^N}{N}\stackrel{(d)} {\rightarrow} 0\quad\mbox{and}\quad
\frac{T_N-\tau_1^N}{N}\stackrel{(d)} {\rightarrow} 0 %
\]
as $N\rightarrow\infty$. By Slutzky's lemma, we can conclude that
$T_N\rightarrow t_0(x)$ in law.
\end{pf}

The following theorem is a key result in the analysis of the fluid
limits of this system. It states that if $\rho_1<1/2$ and $L_1(0)$ is
of the order of $L_2(0)^{\alpha^*}$, there is nonempty time interval
\emph{on the fluid time scale} on which the relation $L_1\sim L_2^{\alpha
^*}$ holds.

\begin{theorem}\label{lem1}
Suppose that $\rho_1<1/2$, $\rho_1+\rho_2<1$ and let $\kappa>0$.
Then there exists $\eta_0>0$ such that for any sequence $(l_1^N)$ satisfying
\[
\limsup_{N\to+\infty} \biggl\llvert \frac{l_1^N}{N^{\alpha^*}}-1\biggr\rrvert
\leq \kappa,
\]
if $(L_1^N(t),L_2^N(t))$ is the solution to the stochastic differential
equation~(\ref{SDE}) with initial condition
$(L_1(0),L_2(0))=(l_1^N,N)$, then
%
\begin{equation}
\label{conc1} \lim_{N\to+\infty}\P \biggl(\sup_{0\leq s\leq\eta_0N }
\biggl\llvert \frac{L_1^N(s)}{(L_2^N(s))^{\alpha^*}} -1\biggr\rrvert > \kappa \biggr)=0.
\end{equation}
\end{theorem}

\begin{pf}
The main argument consists in controlling the upward jumps of $L_1$ on
sufficiently small fluid time scale intervals by an ergodic reflected
random walk. It gives the excursions of $L_1$ on the interval are upper
bounded by the excursions of the reflected random walk and, therefore,
cannot be too large by Kingman's inequality. The same argument applies
to the downward jumps. Since $L_1$ remains close to $L_2^{\alpha^*}$ at
the end of the time interval it gives finally the result.

Let us write $a_0=1+\kappa$, and let us prove that there exists $\eta
_0>0$ such that
%
\begin{equation}
\label{conc} \lim_{N\to+\infty}\P \biggl(\sup_{0\leq s\leq\eta_0 N}
\frac
{L_1^N(s)}{(L_2^N(s))^{\alpha^*}} > a_0 \biggr)=0.
\end{equation}
The other inequality can then be shown by using the same technique, and
so we omit the details.

The process $(L_2^N(s))$ is stochastically bounded from above by
$(\overline{Q}_2^N(s))$ describing the number of class~$2$ jobs in the
priority system where class $1$ jobs have the priority of service, that is,
queue~$2$ is served only when queue~$1$ is empty. For this system, it
is not difficult to show that the convergence in distribution
\[
\lim_{N\to+\infty} \biggl(\frac{\overline{Q}_2^N(Ns)}{N}, s<1 \biggr)=\bigl(1+
\mu_2(\rho_1+\rho_2-1)s, s<1\bigr)
\]
holds. Similarly, the process $(L_2^N(s))$ is stochastically bounded
from below by $(\underline{Q_2^N}(s))$ describing the number of
class~$2$ jobs when they have priority, and one has the corresponding
convergence in distribution
\[
\lim_{N\to+\infty} \biggl(\frac{\underline{Q_2^N}(Ns)}{N}, s<1 \biggr)=\bigl(1+
\mu_2(\rho_2-1)s, s<1\bigr).
\]
Fix $0<\eta_0<1$ such that
\[
a_0 \bigl(1+2\mu_2(\rho_2-1)
\eta_0 \bigr)^{\alpha^*}> 1.
\]
Let also ${\mathcal E}$ be the event defined by
\[
{\mathcal E}= \biggl\{\sup_{0\leq s\leq\eta_0 N}\frac
{L_2^N(s)}{N}\leq
\frac{3}{2},  \inf_{0\leq s\leq\eta_0N}\frac{L_2^N(s)}{N}\geq1+
\frac{3}{2}\mu _2(\rho_2-1)\eta_0
\biggr\}.
\]
The above convergence in distribution results show in particular that,
for any $\varepsilon>0$, there exists $N_0$ such that for every
$N\geq N_0$,
\begin{eqnarray*}
\P \bigl({\mathcal E}^c \bigr)&\leq&\P \biggl(\sup
_{0\leq s\leq\eta
_0 N}\frac{\overline{Q}_2^N(s)}{N}> \frac{3}{2} \biggr)
\\
&&{}+ \P \biggl(\inf_{0\leq s\leq\eta_0 N}\frac{\underline{Q_2^N}(s)}{N}< 1+
\frac
{3}{2}\mu_2(\rho_2-1)\eta_0
\biggr)
\\
&\leq&\varepsilon.
\end{eqnarray*}
By relation~(\ref{eqelem}), the service rate of queue~$1$ at time $s$
is given by
\[
\Delta(s)\stackrel{\mathrm{def.}} {=}\lambda_1+\mu_1
\frac{\log
[(1+L_1^N(s))/(L_2^N(s))^{\alpha^*}]}{({\alpha^*}+1)(({\alpha
^*}+1)\log
(L_2^N(s))+\log[(1+L_1^N(s))/(L_2^N(s))^{\alpha^*}])}.
\]
If for some $y>1$ and some $s<1$ $L_1^N(s)\geq y (L_2^N(s))^{\alpha
^*}$, then
\[
\Delta(s)\geq\lambda_1+\mu_1 \frac{\log(y)}{({\alpha
^*}+1)(({\alpha^*}
+1)\log(L_2^N(s))+\log(y))}.
\]
Furthermore,
%
\begin{eqnarray}\label{ineqq}
\Delta(s) &\geq&\mu_N(y)
\nonumber\\[-8pt]\\[-8pt]\nonumber
&\stackrel{\mathrm{def.}} {=}&
\lambda_1+\mu_1 \frac{\log(y)}{({\alpha^*}+1)(({\alpha^*}+1)(\log(3/2)+\log
(N))+\log(y))}
\end{eqnarray}
holds on the event ${\mathcal E}$.

Denote by $(X_y(u))$ the birth and death process on $\Z$ whose $+1$
(resp., $-1$) jumps have rate $\lambda_1$ [resp., $\mu_N(y)$] and
starting at $0$.

Define
\[
a_1=a_0 \biggl(1+\frac{3}{2}\mu_2(
\rho_2-1)\eta_0 \biggr)^{\alpha
^*}\quad \mbox{and}\quad a_2=\frac{1+a_1}{2}.
\]
Since\vspace*{1pt} $\rho_2<1$, the definition of $\eta_0$ gives that $a_1>1$ and,
therefore, $a_2>1$. Note that $a_0>a_2$. Suppose first that
$l_1^N/N^{\alpha^*}\rightarrow a_0$.

A simple coupling argument gives that the processes $(L_1^N(s))$ and
$(X_{a_2}(s))$ can be constructed so that, on the event ${\mathcal
E}$, the relation
%
\begin{equation}
\label{irancy} L_1^N(s)\leq l_1^N+X_{a_2}(s)
\end{equation}
holds for every $s\leq\inf\{u\dvtx  L_1^N(u)/(L_2^N(u))^{\alpha^*}\leq
a_2\}$.

For every $x\geq0$, Kingman's inequality (see again Kingman~\cite
{Kingman} or relation~(3.3) of Theorem~3.5 in Robert~\cite{Robert})
gives us the estimate
%
\begin{equation}
\label{joblot} \P \Bigl(\sup_{s\geq0} X_{a_2}(s)\geq x
\Bigr)\leq\exp \bigl(-x \bigl(\mu_N(a_2)-
\lambda_1\bigr) \bigr).
\end{equation}
In particular, the random variable
\[
\frac{1}{N^{\alpha^*}} \sup_{s\geq0} X_{a_2}(s)
\]
converges in distribution to $0$ since
\[
\mu_N(a_2)\sim\lambda_1+
\frac{\mu_1\log(a_2)}{({\alpha
^*}+1)^2\log N},
\]
as $N$ goes to infinity. Let
\[
T_N=\inf \bigl\{s>0\dvtx  l_1^N+X_{a_2}(s)
\leq a_2N^{\alpha^*} \bigr\}.
\]
In a similar way as in the proof of Proposition~\ref{tajine}, for
example, a simple drift analysis shows that
\[
\E(T_N)\leq\frac{({\alpha^*}+1)(({\alpha^*}+1)(\log(3/2)+\log
(N))+\log
(a_2))}{\mu_1\log(a_2)} l_1^N,
\]
and consequently
\[
\lim_{N\to+\infty}\P\bigl(T_N>N/(\log N)\bigr)=0.
\]
From relation~(\ref{irancy}), one gets that with probability tending
to $1$, $(L_1^N(t)/N^{\alpha^*})$ does not grow above
$(l_1^N(t)/N^{\alpha^*})$ until the time $T_N$, which itself occurs much
before~$N$. As a consequence, it is enough to prove identity~(\ref{conc}) with the assumption that $L_2^N(0)=N$ and
\[
\lim_{N\to+\infty} \frac{L_1^N(0)}{N^{\alpha^*}}=a_2.
\]
The reflected process of the birth and death process $(X_y(u))$ is
denoted by $(X_y^+(u))$. This is in fact an $M/M/1$ queue with input
rate $\lambda_1$ and service rate $\mu_N(y)$. As before, with
inequality~(\ref{ineqq}), one can construct a coupling such that the relation
%
\begin{equation}
\label{xqq} L_1^N(s)\leq L_1^N(0)+X^+_{a_2}(s),
\end{equation}
holds for every $s\leq\eta_0N$ on the event ${\mathcal E}$.

For every $y>0$, let us define
\[
\tau_y=\inf \bigl\{s\geq0\dvtx  X^+_{a_2}(s)\geq y \bigr\}
\quad\mbox{and}\quad H=\inf \biggl\{s\geq0\dvtx  \frac{L_{1}^N(s)}{(L_2^N(s))^{\alpha^*}}> a_0
\biggr\}.
\]
The proposition will be proved if one shows that $\P(H\leq\eta_0 N)$
converges to $0$ as~$N$ goes to infinity.

On the event ${\mathcal E}$, if $0\leq s\leq\eta_0 N$ is such that
$ L_1^N(s)\geq a_0 (L_2^N(s))^{\alpha^*}$ then
\[
L_1^N(s)\geq a_0 \bigl(1+
\tfrac{3}{2}\mu_2(\rho_2-1)\eta_0
\bigr)^{\alpha^*}N^{\alpha^*}=a_1 N^{\alpha^*}.
\]
Consequently, equation~(\ref{xqq}) and the definition of $a_2$ give
that for every $N\geq N_0$,
%
\begin{equation}
\label{yqq} \P(H\leq\eta_0N)\leq\varepsilon+\P (
\tau_{N^{\alpha^*}
(a_1-1)/2}\leq\eta_0N ).
\end{equation}
But using the same technique as in the first part of the proof of
Proposition~\ref{Fluid}, we can show that the\vspace*{1pt} probability that
$X^+_{a_2}$ reaches $N^{\alpha^*}(a_1-1)/2$ in one excursion away from
$0$ is so low, that the probability that it reaches this height during
one of the $\mathcal{O}(N)$ excursions it does in the interval
$[0,\eta_0 N]$ tends to $0$ as $N$ tends to infinity. Therefore, we
can conclude that
\[
\P(H\leq\eta_0N)\leq2\varepsilon, %
\]
for $N$ large enough. Since this property holds for every $\varepsilon
>0$, Theorem~\ref{lem1} is proved.
\end{pf}

\begin{corollary}\label{corol}
Under the assumptions of Theorem~\ref{lem1}, the convergence in distribution
\[
\lim_{N\to+\infty} \biggl(\frac{L_2^N(Nt)}{N}, 0\leq t<
t_0 \biggr)= \bigl(\gamma(t), 0\leq t< t_0 \bigr)
\]
holds, with $t_0={1}/(\mu_2(1-\rho_1-\rho_2))$ and $\gamma
(t)=1+\mu_2(\rho_1+\rho_2-1)t$.

In addition, for every $t<t_0$ we have
%
\begin{equation}
\label{aqq} \lim_{N\to+\infty}\P \biggl(\sup_{0\leq s\leq t }
\biggl\llvert \frac
{L_1^N(Ns)}{(L_2^N(Ns))^{\alpha^*}} -1\biggr\rrvert > \kappa \biggr)=0.
\end{equation}
\end{corollary}

\begin{pf}
Let us first prove the convergence
%
\begin{equation}
\label{zqq} \lim_{N\to+\infty} \biggl(\frac{L_2^N(Nt)}{N}, 0\leq t
\leq\eta _0 \biggr)= \bigl(\gamma(t), 0\leq t\leq\eta_0
\bigr).
\end{equation}
The SDE~(\ref{SDE}) is used in the same way as before, and so we only
sketch the proof. By Theorem~\ref{lem1}, we have
\[
\lim_{N\to+\infty}\P \biggl(\sup_{0\leq s\leq\eta_0} \biggl
\llvert \frac
{L_1^N(Ns)}{(L_2^N(Ns))^{\alpha^*}}-1\biggr\rrvert \leq\kappa \biggr)=1.
\]
Thus, $L_1^N(Ns)$ is at most of the order of $N^{\alpha^*}$ with
arbitrarily large probability. This implies that for any $s\leq\eta
_0$, all the arrivals at queue~$1$ up to time $Ns$ are processed.
Hence, queue~$1$ uses the fraction $\rho_1$ of the capacity of the
server, and the remaining capacity is devoted to queue~$2$. The
convergence~(\ref{zqq}) is proved. Furthermore, from relation~(\ref
{conc1}) we obtain that
\[
\lim_{N\to+\infty}\P \biggl(\biggl\llvert \frac{L_1^N(N\eta
_0)}{(L_2^N(N\eta_0))^{\alpha^*}}-1\biggr
\rrvert \leq\kappa \biggr)=1,
\]
and $L_2^N(N\eta_0)\sim\gamma(\eta_0)N$. Consequently, Theorem~\ref
{lem1} applied with the initial condition $(L_1^N(N\eta_0),L_2^N(N\eta
_0))$ shows that the convergence~(\ref{zqq}) and relation~(\ref
{conc1}) can be extended to the interval $[0,\eta_0(1+\gamma(\eta
_0))]$. That is,
%
\begin{equation}
\label{conc2} \lim_{N\to+\infty}\P \biggl(\sup_{0\leq s\leq\eta_0(1+\gamma
(\eta_0)) }
\biggl\llvert \frac{L_1^N(Ns)}{(L_2^N(Ns))^{\alpha^*}} -1\biggr\rrvert > \kappa \biggr)=0.
\end{equation}
Proceeding by induction, as long as $\gamma(x)\neq0$ if the
convergence~(\ref{zqq}) and inequality~(\ref{conc2}) hold on $[0,x]$,
these relations can be extended to $[0,x+\gamma(x)\eta_0]$. The
corollary is proved.
\end{pf}
The following theorem is the main result of this section. 
Propositions \ref{Fluid}~and~\ref{tajine} show that if
$L_2^N(0)=N$ then quickly, on a time scale faster than $t\to N^\alpha
\log(N)^2\cdot t$, the first coordinate is very close to $N^{\alpha^*}$.
Since the component $L_2$ does not change much on this time scale, this
can be rephrased as follows: very quickly $L_1\sim L_2^{\alpha^*}$. This
theorem establishes this property on the fluid time scale: $L_2$ is
decreasing linearly on this time scale and $L_1$ adapts very quickly to
the new values of $L_2$ so that $L_1\sim L_2^{\alpha^*}$. This is of
course a much stronger result than Propositions~\ref{Fluid}~and~\ref{tajine}.

\begin{theorem}\label{theoprov}
Suppose that $\rho_1<1/2$ and $\rho_1+\rho_2<1$. If
$(L^N_1(t),L^N_2(t))$ is the solution to the SDE~(\ref{SDE}) with
$L_2^N(0)=N$ and
\[
\lim_{N\to+\infty}\frac{L_1^N(0)}{N^{\alpha^*}}=1,
\]
then we have for the convergence in distribution
\[
\lim_{N\to+\infty} \biggl( \biggl[\frac{L_1^N(Nt)}{N^{\alpha^*}},
\frac{L_2^N(Nt)}{N} \biggr], 0\leq t< t_0 \biggr)= \bigl(\bigl[
\gamma(t)^{\alpha^*},\gamma(t)\bigr], 0\leq t<
t_0 \bigr),
\]
where $t_0={1}/(\mu_2(1-\rho_1-\rho_2))$ and $\gamma(t)=1+\mu
_2(\rho_1+\rho_2-1)t$.
\end{theorem}

\begin{pf}
From Corollary~\ref{corol}, we obtain that relation~(\ref{aqq}) holds
for any $\kappa>0$. Hence, the process
\[
\biggl(\frac{L_1^N(Ns)}{(L_2^N(Ns))^{\alpha^*}}, 0\leq s<t_0 \biggr)
\]
converges\vspace*{1pt} in distribution to the process constant equal to $1$ on
$[0,t_0)$ as $N$ goes to infinity. We conclude with the convergence of
$(L_2^N(Nt)/{N},  0\leq t< t_0)$.
\end{pf}
We can now state a fluid limit result concerning this network under the
assumption $\rho_1<1/2$ and $\rho_2>1/2$ and with the more general
initial conditions considered in Section~\ref{ts1}. Analogous
statements for the other cases are available, but for the sake of
simplicity we do not give them here.

\begin{theorem}[(Fluid limits)]\label{fluidtheo}
Suppose that $\rho_1<1/2$ and $\rho_2>1/2$. If $(L^N_1(t), L^N_2(t))$
is the solution to the SDE~(\ref{SDE}) with initial conditions such that
\[
\lim_{N\to+\infty} \biggl(\frac{L^N_1(0)}{N}, \frac
{L^N_2(0)}{N}
\biggr)=(x,1-x),
\]
for some $x \in[0,1]$ then, for the convergence in distribution, we have
\[
\lim_{N\to+\infty} \biggl(\frac{L_1^N(Nt)}{N},\frac{L_2^N(Nt)}{N}
\biggr)=\bigl(\ell_1(t),\ell_2(t)\bigr),
\]
where the pair $(\ell_1,\ell_2)$ is defined as follows: if
$t_1(x)={2x}/(\mu_1-2\lambda_1)$,
\[
\bigl(\ell_1(t),\ell_2(t)\bigr)= \cases{ \displaystyle
\biggl(x+ \biggl[\lambda_1-\frac{\mu_1}{2} \biggr]t,
(1-x)+ \biggl[\lambda_
2-\frac{\mu_2}{2} \biggr]t \biggr),& \quad$t{\leq}t_1$,
\vspace*{3pt}\cr
\displaystyle \bigl(0, \bigl(
\ell_2(t_1)+ \bigl[\lambda_2-\mu
_2(1-\rho_1) \bigr](t-t_1) \bigr)^+
\bigr),&\quad$t{\geq} t_1$.}
\]
\end{theorem}

\begin{pf}
We give only a sketch of the proof. Until the time $Nt_1$, both queues
are of the same asymptotic order for the $\log$ function.
Consequently, as it has been already seen in the proof of
Proposition~\ref{Jac}, they both receive half of the capacity. Notice
that on the time interval $[0,Nt_1]$, because of our assumptions the
variable $(L_1^N(t))$ decreases whereas $(L_2^N(t))$ increases. After
time $Nt_1$, from Proposition~\ref{Jac} and Theorem~\ref{theoprov},
the variable $(L_1^N(t))$ is of the order of $N^{\alpha^*}$ and is
therefore negligible for the fluid scaling.
\end{pf}

The following proposition completes the description of fluid limits of
the system. Its proof follows the same line as the above proposition,
it is omitted.

\begin{proposition}\label{fluid+}
If $(L^N_1(t),L^N_2(t))$ is the solution to the SDE~(\ref{SDE}) with
initial conditions such that
\[
\lim_{N\to+\infty} \biggl(\frac{L^N_1(0)}{N}, \frac
{L^N_2(0)}{N}
\biggr)=(x,1-x),
\]
then, for the convergence in distribution, we have
\[
\lim_{N\to+\infty} \biggl(\frac{L_1^N(Nt)}{N},\frac{L_2^N(Nt)}{N}
\biggr)=\bigl(\ell_1(t),\ell_2(t)\bigr),
\]
where the pair $(\ell_1,\ell_2)$ is defined as follows:

\begin{itemize}
\item[--] If $\rho_1>1/2$ and $\rho_2>1/2$, then
\[
\bigl(\ell_1(t),\ell_2(t)\bigr)= \biggl(x+ \biggl[
\lambda_1-\frac{\mu
_1}{2} \biggr]t, (1-x)+ \biggl[
\lambda_
2-\frac{\mu_2}{2} \biggr]t \biggr),\qquad t\geq0.
\]
\item[--] If $\rho_1<1/2$ and $\rho_2<1/2$,
when $x$ is such that
\[
t_1\stackrel{\mathit{def.}} {=}\frac{x}{\mu_1/2-\lambda_1}\leq
t_2\stackrel{\mathit{def.}} {=}\frac{1-x}{\mu_2/2-\lambda_2}
\]
then
\[
\bigl(\ell_1(t),\ell_2(t)\bigr)= \cases{ \displaystyle
\biggl(x+ \biggl[\lambda_1-\frac{\mu_1}{2} \biggr]t,
(1-x)+ \biggl[\lambda_
2-\frac{\mu_2}{2} \biggr]t \biggr),&
\quad$t{\leq}t_1$,
\vspace*{3pt}\cr
\displaystyle \bigl(0, \bigl(
\ell_2(t_1)+ \bigl[\lambda_2-\mu
_2(1-\rho_1) \bigr](t-t_1) \bigr)^+
\bigr),&\quad$t{\geq} t_1$,}
\]
and similarly for $t_2\leq t_1$.
\end{itemize}
\end{proposition}

\begin{rem*}
Note that the constant ${\alpha^*}$ does not play a role in the fluid
limit, but clearly enough this result is much weaker that Theorem~\ref
{theoprov}. As we shall see in Section~\ref{heavysec}, the constant
${\alpha^*}$ plays an important qualitative role in the expression of the
invariant distribution when the regime is close to saturation.
\end{rem*}

\section{Heavy traffic regime}\label{heavysec}
When $\rho_1+\rho_2<1$, since the queueing system is work conserving
the Markov process $(L_1(t),L_2(t))$ has an invariant distribution~$\pi
_\rho$, where $\rho=(\rho_1,\rho_2)$. In the following, $(L_{\rho,1},L_{\rho,2})$ stands for a random variable with distribution $\pi
_\rho$. Obtaining explicit expressions to describe $\pi_\rho$ seems
to be quite difficult: because of the logarithmic weights, the double
generating function of $(\pi_\rho(m,n), (m,n)\in\N)$ does not
satisfy an autonomous equation as it is often the case, for example, in
classic two-dimensional reflected random walks. A precise result on the
asymptotic behavior of $\pi_\rho$ when $\rho_1+\rho_2$ is close to
$1$ can nevertheless be obtained. Its proof relies heavily on the
scaling results of the previous sections.

\begin{theorem}[(Heavy traffic limit of invariant distribution)]\label{heavyth}
If $\rho=(\rho_1,\rho_2)$ and $\rho_1<1/2$ is fixed, then as $\bar
{\rho}=\rho_1+\rho_2$ tends to $1$, we have the convergence in distribution
\[
\lim_{\lambda_2 \nearrow\mu_2(1-\rho_1)} \bigl((1-\bar{\rho })^{{\alpha^*}}
L_{\rho,1},(1-\bar{\rho})L_{\rho,2} \bigr)=\bigl(E_\eta
^{\alpha^*},E_\eta\bigr),
\]
where here again ${\alpha^*}=\rho_1/(1-\rho_1)$ and $E_\eta$ denotes
an exponentially distributed random variable whose parameter $\eta$ is
given by
\[
\eta^{-1}=\frac{\mu_2}{\sqrt{2}}\sqrt{1+\frac{(\mu_1-\mu
_2)^2}{\mu_1\mu_2}
\rho_1(1-\rho_1)}.
\]
\end{theorem}

\begin{pf}
The proof uses the fact that the load of this system, that is, the sum
of the services to be processed, is the same as for a classical $M/G/1$
FIFO queue for which a classical heavy traffic limit results holds. It
gives that, for $\bar{\rho}\sim1$, the relation ${L_{\rho,1}}/{\mu
_1}+{L_{\rho,2}}/{\mu_2}\sim X/(1-\bar{\rho})$ holds in
distribution for a convenient exponential random variable. This shows
that at least one of the variables must be large. Then one considers
the process $(L_1(t),L_2(t))$ with initial state given by $(L_{\rho,1},L_{\rho,2})$ so that it is stationary. The results of the previous
sections show that if the initial state is large, very quickly $L_1$
will be of the order of $L_2^{\alpha^*}$, but since the process is
stationary it implies that this is also true for $L_{\rho,1}$ and
$L_{\rho,2}^{\alpha^*}$.

For every $\xi>0$, let $E_{\xi}$ denote, as before, an exponential
random variable with parameter $\xi$. All the variables used here will
be assumed to be independent.
The total workload of the system is independent of the service
allocation and so has the same (invariant) distribution as the workload
of an $M/G/1$ queue, with arrival rate $\lambda_1+\lambda_2$ and with
the same service distribution as
\[
\sigma\stackrel{\mathrm{def.}} {=}BE_{\mu_1}+(1-B)E_{\mu_2},
\]
where $B$ is a Bernoulli random variable with parameter $\lambda
_1/(\lambda_1+\lambda_2)$.
Hence, Kingman's heavy traffic result for the workload at equilibrium
(see Kingman~\cite{Kingman05} or Proposition~3.10 of Robert~\cite{Robert}) gives the convergence in distribution
%
\begin{equation}
\label{King} \lim_{\lambda_2 \nearrow\mu_2(1-\rho_1)} (1-\bar{\rho}) \biggl(
\frac{L_{\rho,1}}{\mu_1}+\frac{L_{\rho,2}}{\mu_2} \biggr)= E_{\eta_0},
\end{equation}
where $\eta_0$ is the constant given by
\begin{eqnarray*}
\eta_0 &=& \lim_{\lambda_2 \nearrow\mu_2(1-\rho_1)} \frac
{2}{(\lambda_1+\lambda_2)\sqrt{\Var(\sigma-E_{\lambda_1+\lambda_2})}}
\\
&=&
\sqrt{2} \Big/\sqrt{1+\frac{(\mu_1-\mu_2)^2}{\mu_1\mu_2}\rho _1(1-\rho_1)}.
\end{eqnarray*}
In particular, the family of random variables
%
\begin{equation}
\label{seq} \bigl[(1-\bar{\rho}) (L_{\rho,1}), (1-\bar{\rho})
(L_{\rho,2}) \bigr]
\end{equation}
is tight as $\bar{\rho}\nearrow1$. Let us denote a possible limit by
$(X_1,X_2)$, corresponding to a sequence $(\bar{\rho}_n)$ converging
to $1$.
For every $n\geq1$, let us define $(L_{\rho_n,1}(t),L_{\rho
_n,2}(t))$ as the process with initial condition $(L_{\rho
_n,1},L_{\rho_n,2})$, which is thus a stationary process.

The strategy of the proof can be described as follows: use the results
of Sections~\ref{ts1} and~\ref{ts3} to prove that, for a fixed time
$t_n$, $L_{\rho_n,2}(t_n)$ is large and $L_{\rho_n,1}(t_n)$ is close
to $(L_{\rho_n,2}(t_n))^{\alpha^*}$. Consequently, by stationarity, the
equivalence $L_{\rho_n,1}\sim L_{\rho_n,2}^{\alpha^*}$ holds and
relation~(\ref{King}) then gives us the desired result.

To start with, let us assume that $\P(X_2=0)>0$. Then there exists
$\delta>0$ with the property that for any $\eta>0$, there exists
$n_0\in\N$ such that for every $n\geq n_0$,
%
\begin{equation}
\label{tgto} \P \biggl(L_{\rho_n,2}\leq\frac{\eta}{1-\bar{\rho}_n} \biggr) \geq
\delta.
\end{equation}
Now, relation~(\ref{King}) tells us that $\P(X_1+X_2>0)=1$.
Consequently, for any $\varepsilon>0$ there exists $\eta_0>0$ and
$n_0$ such that if $n\geq n_0$ and $\eta_1+\eta_2<2\eta_0$,
%
\begin{equation}
\label{m1} \P \biggl(L_{\rho_n,1}\leq\frac{\eta_1}{1-\bar{\rho}_n},
L_{\rho
_n,2}\leq\frac{\eta_2}{1-\bar{\rho}_n} \biggr)\leq\varepsilon.
\end{equation}
Let $s_0>0$ and $\eta_1>0$, and set $t_n=s_0/(1-\bar{\rho}_n)$. For
$n$ large enough, we have
%
\begin{eqnarray}\label{dureuil}
\delta&\leq&\P \biggl(L_{\rho_n,2}\leq\frac{\eta_1}{1-\bar{\rho
}_n}
\biggr) =\P \biggl(L_{\rho_n,2}(t_n)\leq\frac{\eta_1}{1-\bar
{\rho}_n}
\biggr)\nonumber
\\
&\leq& \P \biggl(L_{\rho_n,2}(t_n)\leq\frac{\eta_1}{1-\bar{\rho}_n}, \sup
_{[0,t_n]}L_{\rho_n,2}(s)>\frac{\eta_0}{3(1-\bar{\rho
}_n)} \biggr)+
\varepsilon
\nonumber\\[-8pt]\\[-8pt]\nonumber
&&{}+\P \biggl(L_{\rho_n,2}(t_n)\leq\frac
{\eta_1}{1-\bar{\rho}_n},
\\
&&\hspace*{26pt} \sup
_{[0,t_n]}L_{\rho_n,2}(s)\leq\frac
{\eta_0}{3(1-\bar{\rho}_n)},
L_{\rho_n,1}(0)> \frac{\eta
_0}{1-\bar{\rho}_n} \biggr).\nonumber
\end{eqnarray}
Clearly enough since $(L_2(t))$ decreases at rate at most $\mu_2$, as
$n$ increases the first term on the right-hand side of (\ref{dureuil})
can be made arbitrarily small by choosing $\eta_1$ and $s_0$ in such a
way that
%
{\renewcommand{\theequation}{C$_1$}
\begin{equation}\label{C1}
\eta_1 < \frac{\eta_0}{3}+(\lambda_2-
\mu_2)s_0 = \frac{\eta
_0}{3}- \mu_2(1-
\rho_2)s_0.
\end{equation}}\setcounter{equation}{45}%
The third term in the right-hand side of relation~(\ref{dureuil}) can
be upper bounded by
%
\begin{eqnarray}
&& \P \biggl( L_{\rho_n,1}(0)> \frac{\eta_0}{1-\bar{\rho}_n}, \inf_{[0,t_n]}L_{\rho_n,1}(s)<
\frac{2\eta_0}{3(1-\bar{\rho}_n)} \biggr)\label{chateaulapompe}
\\
&&\qquad {}+\P \biggl(L_{\rho_n,2}(t_n)\leq \frac{\eta_1}{1-\bar{\rho}_n}, \sup
_{[0,t_n]}L_{\rho_n,2}(s)\leq\frac{\eta_0}{3(1-\bar{\rho
}_n)},
\nonumber\\[-8pt]\label{chateaudeau} \\[-8pt]\nonumber
&&\hspace*{54pt} L_{\rho_n,1}(0) > \frac{\eta_0}{1-\bar{\rho}_n}, \inf_{[0,t_n]}L_{\rho_n,1}(s)
\geq\frac{2\eta_0}{3(1-\bar{\rho
}_n)} \biggr).
\end{eqnarray}
As before, the quantity in (\ref{chateaulapompe}) will tend to $0$ as
$n \rightarrow\infty$ if we choose $s_0$ such that
{\renewcommand{\theequation}{C$_2$}
\begin{equation}\label{C2}
\frac{\eta_0}{3}-\mu_1(1-\rho_1)s_0>0.
\end{equation}}%
Finally, if
\[
\sup_{[0,t_n]}L_{\rho_n,2}(s)\leq\frac{\eta_0}{3(1-\bar{\rho
}_n)}\quad
\mbox{and}\quad\inf_{[0,t_n]}L_{\rho_n,1}(s)\geq
\frac
{2\eta_0}{3(1-\bar{\rho}_n)},
\]
the infinitesimal drift $\Delta_{\rho_n,2}(s)$ of $L_{\rho_n,2}$
satisfies on $[0,t_n]$
\begin{eqnarray*}
\Delta_{\rho_n,2}(s) &\geq& \mu_2 \biggl(\rho_2-
\frac{\log (\eta
_0/(3(1-\bar{\rho}_n)) )}{\log (\eta_0/(3(1-\bar{\rho
}_n)) )+\log (2\eta_0/(3(1-\bar{\rho}_n)) )} \biggr)
\\
&\sim& \mu_2 \biggl(\rho_2-\frac{1}{2} \biggr),
\end{eqnarray*}
as $n$ goes to infinity. But $\rho_{n,2}$ converges to $1-\rho_1>1/2$
by assumption, which means that the quantity in (\ref{chateaudeau})
will tend to $0$ as $n$ gets large whenever
%
{\renewcommand{\theequation}{C$_3$}
\begin{equation}\label{C3}
\eta_1<\mu_2(1/2 - \rho_1)s_0.
\end{equation}}\setcounter{equation}{47}%
Choosing first $s_0$ small enough to match (\ref{C1}) and (\ref{C3}), and
then $\eta_1$ small enough to satisfy (\ref{C1}) and (\ref{C3}), we can conclude
that the right-hand side of relation~(\ref{dureuil}) can be made
smaller than $\delta/2$, which yields a contradiction. We have thus
proved that $\P(X_2>0)=1$.

As a consequence, for any $\varepsilon>0$, one can find $\eta_0$ and
$n_0\in\N$ such that for every $n\geq n_0$,
\[
\P \biggl(L_{\rho_n,2}\geq\frac{\eta_0}{1-\bar{\rho}_n} \biggr)\geq1-\varepsilon.
\]
The last step follows an analogous path, in that we choose a convenient
time $t_n$ to first show that, for any $\beta>{\alpha^*}$, the quantity
\[
\P \biggl(L_{\rho_n,2}\geq\frac{\eta_1}{1-\bar{\rho}_n}, L_{\rho
_n,1}\geq
\frac{1}{(1-\bar{\rho}_n)^\beta} \biggr)
\]
is arbitrarily small with the help of Proposition~\ref{tajine}, and
next that for $K$ sufficiently large,
\[
L_{\rho_n,1}\in (L_{\rho_n,2} )^{\alpha^*}+ \bigl[-K\sqrt{
(L_{\rho_n,2} )^{\alpha^*}\log (L_{\rho
_n,2} )},K\sqrt{
(L_{\rho_n,2} )^{\alpha^*}\log (L_{\rho_n,2} )} \bigr]
\]
with high probability by Proposition~\ref{Jac}.

Since $L_{\rho_n,1}$ is of the order of $(L_{\rho_n,2})^{\alpha^*}$ and
that ${\alpha^*}<1$, relation~(\ref{King}) gives the desired
convergence result.
\end{pf}

\section{General functions for resource sharing}\label{loglogsec}\label{Genfsec}
In this section, we consider the case where the function $(\log x)$
describing the access to the resource is replaced by some other
function $(f(x))$. For a two node network, the corresponding $Q$-matrix
is given by: for every $x\in\N_+^2$,
%
\begin{equation}
\label{Qf} \cases{ q(x,x+e_i)=\lambda_i,
\cr
\displaystyle q(x,x-e_i)=\mu_i \frac{f(x_i)}{f(x_1)+f(x_2)}.}
\end{equation}
To concentrate on the most interesting case, throughout this section we
assume that $\rho_1<1/2$. We analyze only the first two time scales in
order to stress the difference with the $\log$ function, at least
concerning the scaling parameters. The proofs of the results presented
here are similar to those in the log case, and so most of the time we
only sketch them.

\subsection{The $\log\log$ function}
\label{subsloglog}
Here, we consider the function $f(x)=\break \log(\log(e+x))$. To simplify
the notation, we use instead the function $x\rightarrow\log
_2(x)\stackrel{\mathrm{def.}}{=}\log(\log x)$. This, of course, does
not change the convergence results obtained in the following paragraphs.

\subsubsection*{Initial phase}
The first time scale is $t\mapsto\phi_N(t)$ with
\[
\phi_N(t)=\exp \bigl[(\log N)^t \bigr].
\]
The stochastic evolution equation is in this case
%
\begin{eqnarray}\label{evol}
L_1\bigl(\phi_N(t)\bigr)
&=&L_1(0)+\lambda_1 \phi_N(t)
\nonumber\\[-8pt]\\[-8pt]\nonumber
&&{} -
\mu_1\int_0^{\phi_N(t)} \frac{\log
_2(L_1(u))}{\log_2(L_1(u))+\log_2(N)}
\,\diff u +M_N\bigl(\phi_N(t)\bigr),
\end{eqnarray}
where $(M_N(t))$ is a local martingale. Since $\phi_N(t)\ll N$ as long
as $t<1$, the second coordinate $L_2$ stays at $N$ at least up to
$t\approx1$. This justifies the fact that in the above expression, the
term $\log_2(L_2(u))$ has been replaced by $\log_2(N)$.

Let us now define $Z_N(t)=L_1(\phi_N(t))/\phi_N(t)$. We have
\begin{eqnarray*}
Z_N(t)&=&Z_N(1)+\lambda_1 -
\frac{1}{\phi_N(t)}
\\
&&{}
- \frac{\mu
_1}{\phi_N(t)}\int_0^{t}
\frac{\log_2(Z_N(u)\phi_N(u))}{\log
_2(Z_N(u)\phi_N(u))+\log_2(N)} \phi_N'(u) \,\diff u +
\frac{M_N(\phi_N(t))}{\phi_N(t)}.
\end{eqnarray*}
Note that, for every $u\geq0$,
\[
\log_2\bigl(Z_N(u)\phi_N(u)\bigr)=
\log_2\bigl(\phi_N(u)\bigr)+ \log \biggl( 1+
\frac
{\log(Z_N(u))}{\log(\phi_N(u))} \biggr)
\]
and
\[
\log_2\bigl(\phi_N(u)\bigr) = u\log_2(N).
\]
Hence, using the same methods as in Section~\ref{ts1}, we obtain the
following equivalence for the convergence in distribution of processes,
uniformly over any compact subset of $(0,{\alpha^*})$:
\[
\bigl(Z_N(t)\bigr)\sim \biggl(\lambda_1 -
\frac{\mu_1}{\phi_N(t)}\int_0^{t}
\frac{u}{u+1} \phi_N'(u) \,\diff u \biggr) \sim
\biggl(\lambda_1 - \mu_1 \frac{t}{t+1} \biggr).
\]
The following proposition is the analogue of Proposition~\ref
{initprop} for the $\log\log$ function. Recall that ${\alpha
^*}={\rho
_1}/{(1-\rho_1)}<1$ since $\rho_1<1/2$.

\begin{proposition}\label{initproploglog}
If $(L^N_1(t),L^N_2(t))$ is the Markov process with $Q$-matrix\break (\ref
{Qf}) and with initial condition $(0,N)$, then the convergence in distribution
\[
\lim_{N\to+\infty} \biggl(\frac{L^N_1 (\exp [(\log
N)^t ] )}{\exp [(\log N)^t ]}, 0<t<{\alpha^*} \biggr)=
\biggl(\lambda_1- \mu_1\frac{t}{t+1}, 0<t<{
\alpha^*} \biggr)
\]
holds for the uniform topology on compact sets of $(0,{\alpha^*})$.
\end{proposition}

\subsubsection*{The local equilibrium}
The last proposition together with the results obtained in Section~\ref
{ts2} suggest that, if $\rho_1<1/2$, the process should remain stable
around the value $\exp [(\log N)^{\alpha^*} ]$. As in
Section~\ref{ts2}, let us assume that $L_1^N(0)=\delta\phi_N({\alpha^*}
)$ for some $\delta\leq1$, while $L_2^N(0)=N$. Using the relation
$\rho_1={\alpha^*}/(1+{\alpha^*})$ and the evolution equation~(\ref
{evol}), we obtain that for every $t\geq0$
\begin{eqnarray*}
L_1^N( t)&=&L_1^N(0)
\\
&&{}  - \frac{\mu_1}{{\alpha^*}+1}\int_0^{t}
\frac
{\log
[\log(L_1^N(u))/(\log N )^{\alpha^*}]}{\log[\log(L_1^N(u))/(\log
N)^{\alpha^*}]+({\alpha^*}+1)\log_2(N)} \,\diff u
\\
&&{}+M_N(t).
\end{eqnarray*}
As we shall see, the appropriate scaling of time around $\phi
_N({\alpha^*})$ turns out to be $t\mapsto\psi_N t$, where
\[
\psi_N \stackrel{\mathrm{def.}} {=}\phi_N\bigl({
\alpha^*}\bigr) (\log N )^{\alpha^*}\log_2(N)=\exp \bigl[(\log
N)^{\alpha^*} \bigr] (\log N )^{\alpha^*}\log\log(N).
\]
Indeed, let $\widehat{Z}_N(t)=L_1^N(\psi_N t)/\phi_N({\alpha^*})$. For
every $u>0$, we have
\[
\log \biggl[\frac{\log(L_1^N(\psi_N u))}{(\log N)^{\alpha^*}} \biggr] =\log \biggl[1+ \frac{\log(\widehat{Z}_N(u))}{(\log N)^{\alpha^*}
}
\biggr],
\]
so that
\begin{eqnarray*}
&& \widehat{Z}_N(t)-\widehat{Z}_N(0)-\frac{M_N(\psi_N t)}{\phi
_N({\alpha^*})}+
\frac{M_N(1)}{\phi_N({\alpha^*})}
\\
&&\qquad = -\frac{\mu_1}{({\alpha^*}+1)\phi_N({\alpha^*})}
\\
&&\hspace*{39pt}{}\times \int_0^{\psi_N t}
\frac{\log[\log(L_1^N(u))/(\log N)^{\alpha^*}]}{\log[\log
(L_1^N(u))/(\log N)^{\alpha^*}]+({\alpha^*}+1)\log_2(N)} \,\diff u
\\
&&\qquad \sim- \frac{\mu_1  (\log N )^{\alpha^*}\log
_2(N)}{({\alpha^*}+1)} \int_0^{t}
\frac{\log(\widehat{Z}_N(u))}{\log(\widehat
{Z}_N(u))+({\alpha^*}+1) (\log N )^{\alpha^*}\log_2(N)} \,\diff u.
\end{eqnarray*}
With the same tightness argument as in the proof of Proposition~\ref
{Fluid}, we obtain the corresponding convergence result detailed below.
It is remarkable that the limit is the same as in Proposition~\ref{Fluid}.

\begin{proposition}\label{Fluidll}
If $(L^N_1(t),L^N_2(t))$ is the Markov process with $Q$-mat\-rix~(\ref
{Qf}), and initial conditions $L_2^N(0)=N$ and $L_1^N(0)\sim\delta
\exp[(\log N)^{\alpha^*}]$ for some $\delta\in(0,1]$, then for the
convergence in distribution of processes we have
\[
\lim_{N\to+\infty} \bigl(L_1^N(
\psi_N t)e^{-(\log N)^{\alpha^*}
} \bigr) = \bigl(h(t)\bigr),
\]
where
\[
\psi_N=\exp \bigl[(\log N)^{\alpha^*} \bigr] (\log N
)^{\alpha^*}\log\log N
\]
and $(h(t))$ is the function defined by relation~(\ref{eqh}).

Furthermore, if $L_1^N(0)\sim\exp[(\log N)^{\alpha^*}]+y \sqrt{\psi
_N}$ for some $y\in\R$,
then the sequence of processes
\[
\biggl(\frac{L_1^N(\psi_N t)- e^{(\log N)^{\alpha^*}}}{\sqrt{\psi
_N}} \biggr)
\]
converges in distribution to the Ornstein--Uhlenbeck process defined by
relation~(\ref{OU}).
\end{proposition}

\begin{rem*}
If $\rho_1<1/2$ and if the initial condition is $(0,N)$, we obtain
that $L_1$ is of the order of
\[
\exp\bigl[(\log N)^{\alpha^*}\bigr]
\]
on the time scale $t\mapsto\psi_N t$. This is much smaller that the
quantity $N^{\alpha^*}$ corresponding to the $\log$ policy.
One can then take a function $f$ growing more slowly to infinity, such
as $\log\log\log$. Under the same assumptions, the variable $L_1^N$
live in a region with an even smaller order of magnitude. By pushing
this scheme a little further, we would obtain a policy similar to the
Head of the Line Processor-Sharing (see Bramson~\cite{Bramson2}),
where node $j$ receives the bandwidth
\[
\frac{\mathbf{1}_{\{n_i\neq0\}}}{\mathbf{1}_{\{n_1\neq0\}}+\cdots
+\mathbf{1}_{\{n_J\neq0\}}}.
\]
The main drawback of this policy is that a node with many jobs has the
same fraction of the capacity as a node with only a few of them. For
this reason, it is more likely that local congestion will occur more frequently.
\end{rem*}

\subsection{The time scales for a general function}
Finally, let us return to the general case. Let us assume that the
function $x\mapsto f(x)$ is an increasing continuous function on $\R
_+$, tending to infinity as $x\rightarrow\infty$, and suppose that
there exist two functions $x\mapsto A_f(x)$ and $x\mapsto B_f(x)$ on
$\R_+$ such that for any $t>0$ and $z\geq0$,
%
{\renewcommand{\theequation}{F$_1$}
\begin{equation}\label{F1}
\lim_{x\to+\infty}\frac{f^{-1}(tf(x))}{x}=0,\qquad \lim
_{x\to+\infty}\frac{f(zx)-f(x)}{A_f(x)}=B_f(z).
\end{equation}}\setcounter{equation}{49}%
The first time scale for the initial phase is given by
%
\begin{equation}
\label{tsf1} \phi_N\dvtx  t\mapsto f^{-1}\bigl(tf(N)\bigr),
\end{equation}
where $f^{-1}$ denotes the inverse function of $f$. If
$(L^N_1(t),L^N_2(t))$ is the Markov process with $Q$-matrix~(\ref{Qf})
and initial condition $(0,N)$, then the convergence in distribution
\[
\lim_{N\to+\infty} \biggl(\frac{L^N_1 (\phi_N(t) )}{\phi
_N(t)}, 0<t<{\alpha^*} \biggr)=
\biggl(\lambda_1- \mu_1\frac
{t}{t+1}, 0<t<{
\alpha^*} \biggr)
\]
holds for the uniform topology on compact sets of $(0,{\alpha^*})$.
Indeed, the first relation in condition~(\ref{F1}) ensures that the second
coordinate $L_2^N$ stays at $N$ while $L_1^N$ is of the order of $\phi
_N(t)$, so that exactly the same techniques as in Section~\ref{ts1}
can be applied.

The second time scale of interest is $t\mapsto\psi_N t$, where
%
\begin{equation}
\label{tsf2} \psi_N=\frac{\phi_N({\alpha^*})f(N)}{A_f(\phi_N({\alpha^*}))}.
\end{equation}
If $(L^N_1(t),L^N_2(t))$ is the Markov process with $Q$-matrix~(\ref
{Qf}) and initial conditions $L_2(0)=N$ and $L_1(0)\sim\delta\phi
_N({\alpha^*})$ for some $\delta\in(0,1]$, then we conjecture that the
convergence in distribution of processes
\[
\lim_{N\to+\infty} \biggl(\frac{L_1^N(\psi_N t)}{\phi_N({\alpha^*}
)} \biggr) = \bigl(h(t)
\bigr),
\]
should hold, where $(h(t))$ is the function defined by
%
\begin{equation}
\label{eqhf} \cases{ h\equiv1,&\quad if $\delta=1$,
\vspace*{3pt}\cr
\displaystyle\int
_{\delta}^{h(t)}\frac{1}{B_f(u)} \,\diff u=-
\frac
{\mu t}{(1+{\alpha^*})^2}, &\quad if $\delta\neq1$.}
\end{equation}
Some regularity properties (required to use Lebesgue's differentiation
theorem, e.g.) are clearly necessary to justify this
convergence, but then the rest of the proof should follow the lines of
the proof of the corresponding result for the $\log$ case.

\subsubsection*{Some examples}
\begin{longlist}[(a)]
\item[(a) $f(x)=\log\log(x)$.]
In this case, $A_f(x)=1/\log x$ and $B_f(x)=\log x$, and we recover the
expression of the time scales obtain in Section~\ref{subsloglog} for
the $\log\log$ function, that is,
\[
\phi_N(t)= \exp\bigl[(\log N)^t\bigr]\quad\mbox{and}
\quad \psi_N=(\log N)^{\alpha^*}\log\log(N) \exp \bigl((\log
N)^{\alpha^*} \bigr).
\]
\item[(b) $f(x)=(\log x)^\beta$.]
Here, $A_f(x)=1/[\beta(\log x)^{\beta-1}]$ and $B_f(x)=\log x$, we
gives the two time scales
\[
\phi_N(t)= N^{t^{1/\beta}}\quad\mbox{and}\quad \psi_N=
\beta{\alpha^*}^{1-1/\beta} N^{{\alpha^*}^{1/\beta}}(\log N)^{2\beta-1}.
\]
In particular, we recover the time scales of the $\log$ case which
have already been identified.
\end{longlist}
Note that the functions $x\mapsto x^\alpha$, $\alpha>0$, do not
satisfy the first relation in condition~(\ref{F1}).

\section{The network with $J>2$ nodes}\label{Multisec}
In this section, we briefly describe the case of $J>2$ nodes competing
for the single resource. A special case is considered to illustrate the
similarities and also the differences in qualitative behaviors. The
motivation of this section is to show that the analysis of the two node
network gives the main ideas to start the investigation of more
complicated situations. New difficulties are indicated in the text, and
the proofs of these results will be the subject of a further work in a
more general setting.

Let us assume that
\[
\rho_1<\rho_2<\cdots<\rho_{J}\quad\mbox{and}
\quad\sum_{1}^{J}\rho_j<1.
\]
If the state of the system is $(n_j)_{1\leq j\leq J}$, the $k$th
station receives the fraction of service
\[
\frac{\log(1+n_k)}{\log(1+n_1)+\log(1+n_2)+\cdots+\log(1+n_{J})}.
\]
From now on, we consider the case where $L_j(0)=0$ for $j=1,\ldots,
J-1$, and $L_{J}(0)=N$. As before $(L^N_j(t))$ denotes the Markov
process describing the number of requests in each queue and with this
initial condition.

\subsection{Initial phase}
The following proposition is analogous to Proposition~\ref{initprop},
and can be proved in the same way.

\begin{proposition}\label{initpropJ}
If $(L^N_j(t))$ is the solution of the SDE~(\ref{SDE}) with the
initial condition $L_j(0)=0$ for $1\leq j\leq J-1$, and $L_{J}=N$, and if
\[
t_1\stackrel{\mathit{def.}} {=}\frac{\rho_1}{1-(J-1)\rho_1} <1,\qquad\mbox{i.e., } \rho_1<\frac{1}{J},
\]
then, for every $1\leq j\leq J-1$, the convergence
\[
\lim_{N\to+\infty} \biggl(\frac{L^N_j(N^t)}{N^t}, 0<t<t_1
\biggr)= \biggl(\lambda_j- \mu_j\frac{t}{(J-1)t+1},
0<t<t_1 \biggr)
\]
holds for the uniform topology on compact sets of $(0,t_1)$.
\end{proposition}

Observe that the quantity $t_1$ is precisely the first time $t$ for which
\[
\lambda_1=\mu_1\frac{t}{1+(J-1) t}.
\]
%
\subsection{Second phase}
Let us now give a more heuristic description of the evolution of the
network after ``time'' $N^{t_1}$, but still on the time scale $t\mapsto
N^t$. As we shall see, for the second phase the exponent in $N$ of the
random variables $(L_j(N^t),2\leq j\leq J-1)$ are still $t$, but the
exponent $\alpha_{2,1}(t)$ of $L_1^N$ becomes a linear function of $t$
with slope less than $1$.

Let us use the different phenomena observed in the two node case to
infer the behavior of the $J$-node system.

First, since $\rho_j>\rho_1$ for every $j\geq2$, the infinitesimal
drift of $L_j(N^t)$ remains positive at least for a small amount of
time after $t_1$. Hence, one should have the following convergence in
distribution: for $2\leq j\leq J-1$,
\begin{eqnarray*}
&& \lim_{N\to+\infty} \biggl(\frac{L^N_j(N^t)}{N^t}, t_1<t<t_2
\biggr)
\\
&&\qquad = \biggl(\lambda_j- \mu_j\frac{t}{\alpha_{1,2}(t)+(J-2)t+1},
t_1<t<t_2 \biggr),
\end{eqnarray*}
where $t_2$ is the first time $t$ at which
\[
\lambda_2=\mu_2\frac{t}{\alpha_{1,2}(t)+(J-2)t+1}.
\]
Second, the station $1$ should remain at a local equilibrium in the
sense that
the coefficient $\alpha_{1,2}(t)$ should be determined as follows:
\[
\lambda_1=\mu_1\frac{\alpha_{1,2}(t)}{\alpha_{1,2}(t)+(J-2)t+1}.
\]
Combining these relations, we obtain that
\[
\alpha_{1,2}(t)=\frac{\rho_1}{1-\rho_1}\bigl(1+(J-2)t\bigr)\quad\mbox {and}
\quad t_2=\frac{\rho_2}{1-\rho_1-(J-2)\rho_2}.
\]
Of course, $t_2$ has to be strictly less than $1$.

\subsection{Subsequent phases}
Let us now give a, still heuristic, description of the $k$th phase for
$1< k<J-1$. The first $k-1$ stations are at a ``local'' equilibrium.
Denoting the exponent of the $j$th station in the $k$th phase by
$\alpha_{j,k}(t)$, the equilibrium is characterized by the relation
\[
\frac{\alpha_{j,k}(t)}{\alpha_{1,k}(t)+\alpha_{2,k}(t)+\cdots
+\alpha_{k-1,k}(t)+(J-k)t+1}=\rho_j
\]
for $1\leq j\leq k-1$. The end of the $k$th phase, $t_{k}$, corresponds
to the situation when the $k$th station reaches an equilibrium, that
is, $t_k$ satisfies
\[
\frac{t_{k}}{\alpha_{1,k}(t_{k})+\alpha_{2,k}(t_{k})+\cdots+\alpha
_{k-1,k}(t_{k})+(J-k)t_{k}+1}=\rho_{k}.
\]
Consequently, we obtain that
\[
\alpha_{j,k}(t)=\frac{\rho_j}{1-\sum_{i=1}^{k-1} \rho_i}\bigl(1+(J-k) t\bigr), \qquad1\leq
j\leq k-1
\]
and
\[
t_{k}=\frac{\rho_{k}}{1-\sum_{i=1}^{k-1}\rho_i-(J-k)\rho_{k}}.
\]
The time $t_{k}$ is such that $t_{k}<1$ if
\[
\sum_{i=1}^{k-1}\rho_i+
(J-k+1)\rho_{k}<1.
\]
%
\subsection{Final phase}
Provided that $t_{J-1}$ is strictly less that $1$, at the time
$N^{t_{J-1}}$ all the stations are at a local equilibrium around
$N^{\alpha_{j,J}}$ where $\alpha_{j,J}$ is given by
\[
\alpha_{j,J}=\frac{\rho_j}{1-\sum_{i=1}^{J-1} \rho_i}.
\]
Keep in mind that strictly speaking, the local equilibrium of the $j$th
station is on the time scale $t\mapsto N^{\alpha_{j,J}}\log N t$.
Hence, the whole process can be thought as a collection of stationary
processes evolving on different time scales. See Figure~\ref{Brown}.

\begin{figure}

\includegraphics{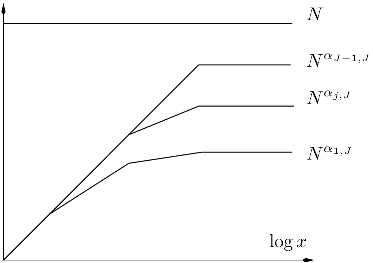}

\caption{The network with $J>2$ nodes on the $t\mapsto N^t$ time scale
when $t_{J-1}<1$.}\label{Brown}
\end{figure}

The difficult technical problem to solve here is for the $J$ phases in
between, during which some of the exponents depend on time, adapting to
the linear growth of the other exponents. One of our main problems
throughout this work has been that we did not succeed in proving
convergence of the $\log(L_j)/\log N$ variables without showing a more
demanding result, namely a convergence result for the variables $L_j$.
This difficulty is even more serious here since the exponents depend on time.


\section*{Acknowledgments}
P.~Robert would like to thank Damon Wischik, whose presentation
at the ICMS workshop in Edinburgh in 2010 is one of the motivations at
the origin of this work.
We thank the reviewer for the detailed work she/he has done on a first
version of the paper.


%

\printaddresses

\begin{thebibliography}{31}
\bibitem{Abramson}
%
\begin{bincollection}[auto:STB|2014/08/04|07:23:14]
\bauthor{\bsnm{Abramson},~\bfnm{N.}\binits{N.}}
(\byear{1970}).
\btitle{The Aloha system}.
In \bbooktitle{FJCC AFIPS Conf. Proc. (Montvale, NJ)}
\bvolume{37}
\bpages{281--285}.
\bpublisher{AFIPS Press}, \blocation{Houston, Texas}.
\end{bincollection}
%
\bptok{imsref}%
\endbibitem

\bibitem{AAU}
%
\begin{barticle}[mr]
\bauthor{\bsnm{Altman},~\bfnm{Eitan}\binits{E.}},
\bauthor{\bsnm{Avrachenkov},~\bfnm{Konstantin}\binits{K.}} \AND
\bauthor{\bsnm{Ayesta},~\bfnm{Urtzi}\binits{U.}}
(\byear{2006}).
\btitle{A survey on discriminatory processor sharing}.
\bjournal{Queueing Syst.}
\bvolume{53}
\bpages{53--63}.
\bid{doi={10.1007/s11134-006-7586-8}, issn={0257-0130}, mr={2230013}}
\end{barticle}
%
\bptok{imsref}%
\endbibitem

\bibitem{BenTahar}
%
\begin{barticle}[mr]
\bauthor{\bsnm{Ben Tahar},~\bfnm{Abdelghani}\binits{A.}} \AND
\bauthor{\bsnm{Jean-Marie},~\bfnm{Alain}\binits{A.}}
(\byear{2012}).
\btitle{The fluid limit of the multiclass processor sharing queue}.
\bjournal{Queueing Syst.}
\bvolume{71}
\bpages{347--404}.
\bid{doi={10.1007/s11134-012-9287-9}, issn={0257-0130}, mr={2945640}}
\end{barticle}
%
\bptok{imsref}%
\endbibitem

\bibitem{Billingsley}
%
\begin{bbook}[mr]
\bauthor{\bsnm{Billingsley},~\bfnm{Patrick}\binits{P.}}
(\byear{1999}).
\btitle{Convergence of Probability Measures},
\bedition{2nd} ed.
\bpublisher{Wiley},
\blocation{New York}.
\bid{doi={10.1002/9780470316962}, mr={1700749}}
\end{bbook}
%
\bptok{imsref}%
\endbibitem

\bibitem{Bonald}
%
\begin{bincollection}[auto:STB|2014/08/04|07:23:14]
\bauthor{\bsnm{Bonald},~\bfnm{Thomas}\binits{T.}} \AND
\bauthor{\bsnm{Massouli{\'e}},~\bfnm{Laurent}\binits{L.}}
(\byear{2001}).
\btitle{Impact of fairness on Internet performance}.
In \bbooktitle{Proceedings of the 2001 ACM SIGMETRICS International
Conference on Measurement and Modeling of Computer Systems (New York,
NY, USA)}
\bpages{82--91}.
\bpublisher{ACM}, \blocation{New York}.
\end{bincollection}
%
\bptok{imsref}%
\endbibitem

\bibitem{Bonald2}
%
\begin{barticle}[mr]
\bauthor{\bsnm{Bonald},~\bfnm{T.}\binits{T.}},
\bauthor{\bsnm{Massouli{\'e}},~\bfnm{L.}\binits{L.}},
\bauthor{\bsnm{Prouti{\`e}re},~\bfnm{A.}\binits{A.}} \AND
\bauthor{\bsnm{Virtamo},~\bfnm{J.}\binits{J.}}
(\byear{2006}).
\btitle{A queueing analysis of max-min fairness, proportional fairness
and balanced fairness}.
\bjournal{Queueing Syst.}
\bvolume{53}
\bpages{65--84}.
\bid{doi={10.1007/s11134-006-7587-7}, issn={0257-0130}, mr={2230014}}
\end{barticle}
%
\bptok{imsref}%
\endbibitem

\bibitem{BBLP}
%
\begin{bincollection}[auto:STB|2014/08/04|07:23:14]
\bauthor{\bsnm{Bouman},~\bfnm{Nick}\binits{N.}},
\bauthor{\bsnm{Borst},~\bfnm{Sem}\binits{S.}},
\bauthor{\bsnm{van~Leeuwaarden},~\bfnm{Johan}\binits{J.}} \AND
\bauthor{\bsnm{Prouti{\`e}re},~\bfnm{Alexandre}\binits{A.}}
(\byear{2011}).
\btitle{Backlog-based random access in wireless networks: Fluid limits
and delay issues}.
In \bbooktitle{23rd International Teletraffic Congress (ITC),
September 2011}
\bpages{39--46}.
\bpublisher{IEEE Computer Society}, \blocation{San Francisco, CA}.
\end{bincollection}
%
\bptok{imsref}%
\endbibitem

\bibitem{Bramson2}
%
\begin{barticle}[mr]
\bauthor{\bsnm{Bramson},~\bfnm{Maury}\binits{M.}}
(\byear{1996}).
\btitle{Convergence to equilibria for fluid models of head-of-the-line
proportional processor sharing queueing networks}.
\bjournal{Queueing Syst.}
\bvolume{23}
\bpages{1--26}.
\bid{doi={10.1007/BF01206549}, issn={0257-0130}, mr={1433762}}
\end{barticle}
%
\bptok{imsref}%
\endbibitem

\bibitem{Bramson}
%
\begin{barticle}[mr]
\bauthor{\bsnm{Bramson},~\bfnm{Maury}\binits{M.}}
(\byear{2008}).
\btitle{Stability of queueing networks}.
\bjournal{Probab. Surv.}
\bvolume{5}
\bpages{169--345}.
\bid{doi={10.1214/08-PS137}, issn={1549-5787}, mr={2434930}}
\end{barticle}
%
\bptok{imsref}%
\endbibitem

\bibitem{Dai}
%
\begin{barticle}[mr]
\bauthor{\bsnm{Dai},~\bfnm{J.~G.}\binits{J.~G.}}
(\byear{1996}).
\btitle{A fluid limit model criterion for instability of multiclass
queueing networks}.
\bjournal{Ann. Appl. Probab.}
\bvolume{6}
\bpages{751--757}.
\bid{doi={10.1214/aoap/1034968225}, issn={1050-5164}, mr={1410113}}
\end{barticle}
%
\bptok{imsref}%
\endbibitem

\bibitem{Ethier01}
%
\begin{bbook}[mr]
\bauthor{\bsnm{Ethier},~\bfnm{Stewart~N.}\binits{S.~N.}} \AND
\bauthor{\bsnm{Kurtz},~\bfnm{Thomas~G.}\binits{T.~G.}}
(\byear{1986}).
\btitle{Markov Processes: Characterization and Convergence}.
\bpublisher{Wiley},
\blocation{New York}.
\bid{doi={10.1002/9780470316658}, mr={0838085}}
\end{bbook}
%
\bptok{imsref}%
\endbibitem

\bibitem{Freidlin02}
%
\begin{bbook}[mr]
\bauthor{\bsnm{Freidlin},~\bfnm{M.~I.}\binits{M.~I.}} \AND
\bauthor{\bsnm{Wentzell},~\bfnm{A.~D.}\binits{A.~D.}}
(\byear{1984}).
\btitle{Random Perturbations of Dynamical Systems},
\bedition{2nd} ed.
\bpublisher{Springer},
\blocation{New York}.
\bnote{Translated from the Russian by Joseph Sz\"ucs}.
\bid{doi={10.1007/978-1-4684-0176-9}, mr={0722136}}
\bptnote{check year}%
\end{bbook}
%
\bptok{imsref}%
\endbibitem

\bibitem{GBW}
%
\begin{bincollection}[auto:STB|2014/08/04|07:23:14]
\bauthor{\bsnm{Ghaderi},~\bfnm{Javad}\binits{J.}},
\bauthor{\bsnm{Borst},~\bfnm{Sem}\binits{S.}} \AND
\bauthor{\bsnm{Whiting},~\bfnm{Phil}\binits{P.}}
(\byear{2012}).
\btitle{Backlog-based random access in wireless networks: Fluid limits
and instability issues}.
In \bbooktitle{46th Annual Conference on Information Sciences and
Systems (CISS), March 2012}
\bpages{1--6}.
\bpublisher{IEEE Computer Society}, \blocation{Princeton, NC}.
\end{bincollection}
%
\bptok{imsref}%
\endbibitem

\bibitem{Graham}
%
\begin{barticle}[mr]
\bauthor{\bsnm{Graham},~\bfnm{Carl}\binits{C.}} \AND
\bauthor{\bsnm{Robert},~\bfnm{Philippe}\binits{P.}}
(\byear{2009}).
\btitle{Interacting multi-class transmissions in large stochastic networks}.
\bjournal{Ann. Appl. Probab.}
\bvolume{19}
\bpages{2334--2361}.
\bid{doi={10.1214/09-AAP614}, issn={1050-5164}, mr={2588247}}
\end{barticle}
%
\bptok{imsref}%
\endbibitem

\bibitem{Khasminski01}
%
\begin{bbook}[mr]
\bauthor{\bsnm{Has'minski{\u\i}},~\bfnm{R.~Z.}\binits{R.~Z.}}
(\byear{1980}).
\btitle{Stochastic Stability of Differential Equations}.
\bpublisher{Sijthoff \& Noordhoff},
\blocation{Alphen aan den Rijn---Germantown, MD}.
\bnote{Translated from the Russian by D. Louvish}.
\bid{mr={0600653}}
\end{bbook}
%
\bptok{imsref}%
\endbibitem

\bibitem{Walrand}
%
\begin{barticle}[auto:STB|2014/08/04|07:23:14]
\bauthor{\bsnm{Jeonghoon},~\bfnm{Mo}\binits{M.}} \AND
\bauthor{\bsnm{Walrand},~\bfnm{Jean}\binits{J.}}
(\byear{2000}).
\btitle{Fair end-to-end window-based congestion control}.
\bjournal{IEEE Trans. Netw.}
\bvolume{8}
\bpages{556--567}.
\end{barticle}
%
\bptok{imsref}%
\endbibitem

\bibitem{Kelly}
%
\begin{barticle}[auto:STB|2014/08/04|07:23:14]
\bauthor{\bsnm{Kelly},~\bfnm{F.}\binits{F.}},
\bauthor{\bsnm{Maulloo},~\bfnm{A.}\binits{A.}} \AND
\bauthor{\bsnm{Tan},~\bfnm{D.}\binits{D.}}
(\byear{1998}).
\btitle{Rate control in communication networks: Shadow prices,
proportional fairness and stability}.
\bjournal{J.~Oper. Res. Soc.}
\bvolume{49},
\bpages{237--252}.
\end{barticle}
%
\bptok{imsref}%
\endbibitem

\bibitem{Kingman05}
%
\begin{bincollection}[mr]
\bauthor{\bsnm{Kingman},~\bfnm{J.~F.~C.}\binits{J.~F.~C.}}
(\byear{1965}).
\btitle{The heavy traffic approximation in the theory of queues.
({W}ith discussion.)}
In \bbooktitle{Proc. {S}ympos. {C}ongestion {T}heory ({C}hapel {H}ill,
NC, 1964)}
\bpages{137--169}.
\bpublisher{Univ. North Carolina Press},
\blocation{Chapel Hill, NC}.
\bid{mr={0198566}}
\bptnote{check related}%
\end{bincollection}
%
\bptok{imsref}%
\endbibitem

\bibitem{Kingman}
%
\begin{barticle}[mr]
\bauthor{\bsnm{Kingman},~\bfnm{J.~F.~C.}\binits{J.~F.~C.}}
(\byear{1970}).
\btitle{Inequalities in the theory of queues}.
\bjournal{J. Roy. Statist. Soc. Ser. B}
\bvolume{32}
\bpages{102--110}.
\bid{issn={0035-9246}, mr={0266333}}
\end{barticle}
%
\bptok{imsref}%
\endbibitem

\bibitem{Malyshev}
%
\begin{barticle}[mr]
\bauthor{\bsnm{Malyshev},~\bfnm{V.~A.}\binits{V.~A.}}
(\byear{1993}).
\btitle{Networks and dynamical systems}.
\bjournal{Adv. in Appl. Probab.}
\bvolume{25}
\bpages{140--175}.
\bid{doi={10.2307/1427500}, issn={0001-8678}, mr={1206537}}
\end{barticle}
%
\bptok{imsref}%
\endbibitem

\bibitem{MR}
%
\begin{bincollection}[auto:STB|2014/08/04|07:23:14]
\bauthor{\bsnm{Massouli{\'e}},~\bfnm{Laurent}\binits{L.}} \AND
\bauthor{\bsnm{Roberts},~\bfnm{James}\binits{J.}}
(\byear{1999}).
\btitle{Bandwidth sharing: Objectives and algorithms}.
In \bbooktitle{INFOCOM'99. Eighteenth Annual Joint Conference of the
IEEE Computer and Communications Societies}
\bpages{1395--1403}.
\bpublisher{IEEE Computer Society}, \blocation{New York}.
\end{bincollection}
%
\bptok{imsref}%
\endbibitem

\bibitem{Metcalf}
%
\begin{barticle}[auto:STB|2014/08/04|07:23:14]
\bauthor{\bsnm{Metcalf},~\bfnm{R.}\binits{R.}} \AND
\bauthor{\bsnm{Boggs},~\bfnm{D.}\binits{D.}}
(\byear{1976}).
\btitle{Ethernet: Distributed packet switching for local computer networks}.
\bjournal{Communications of the ACM}
\bvolume{19}
\bpages{395--403}.
\end{barticle}
%
\bptok{imsref}%
\endbibitem

\bibitem{PSV}
%
\begin{bincollection}[mr]
\bauthor{\bsnm{Papanicolaou},~\bfnm{G.~C.}\binits{G.~C.}},
\bauthor{\bsnm{Stroock},~\bfnm{D.}\binits{D.}} \AND
\bauthor{\bsnm{Varadhan},~\bfnm{S.~R.~S.}\binits{S.~R.~S.}}
(\byear{1977}).
\btitle{Martingale approach to some limit theorems}.
In \bbooktitle{Papers from the {D}uke {T}urbulence {C}onference
({D}uke {U}niv., {D}urham, NC, 1976), {P}aper {N}o. 6,
Duke Univ. Math. Ser., Vol. III}
\bpages{ii+120 pp}.
\bpublisher{Duke Univ.},
\blocation{Durham, NC}.
\bid{mr={0461684}}
\end{bincollection}
%
\bptok{imsref}%
\endbibitem

\bibitem{RR}
%
\begin{barticle}[mr]
\bauthor{\bsnm{Ramanan},~\bfnm{Kavita}\binits{K.}} \AND
\bauthor{\bsnm{Reiman},~\bfnm{Martin~I.}\binits{M.~I.}}
(\byear{2003}).
\btitle{Fluid and heavy traffic diffusion limits for a generalized
processor sharing model}.
\bjournal{Ann. Appl. Probab.}
\bvolume{13}
\bpages{100--139}.
\bid{doi={10.1214/aoap/1042765664}, issn={1050-5164}, mr={1951995}}
\end{barticle}
%
\bptok{imsref}%
\endbibitem

\bibitem{Robert}
%
\begin{bbook}[mr]
\bauthor{\bsnm{Robert},~\bfnm{Philippe}\binits{P.}}
(\byear{2003}).
\btitle{Stochastic Networks and Queues},
\bedition{French} ed.
\bseries{Stochastic Modelling and Applied Probability}
\bvolume{52}.
\bpublisher{Springer},
\blocation{Berlin}.
\bid{doi={10.1007/978-3-662-13052-0}, mr={1996883}}
\end{bbook}
%
\bptok{imsref}%
\endbibitem

\bibitem{Rudin}
%
\begin{bbook}[mr]
\bauthor{\bsnm{Rudin},~\bfnm{Walter}\binits{W.}}
(\byear{1987}).
\btitle{Real and Complex Analysis},
\bedition{3rd} ed.
\bpublisher{McGraw-Hill},
\blocation{New York}.
\bid{mr={0924157}}
\end{bbook}
%
\bptok{imsref}%
\endbibitem

\bibitem{Rybko}
%
\begin{barticle}[auto]
\bauthor{\bsnm{Rybko},~\bfnm{A.~N.}\binits{A.~N.}} \AND
\bauthor{\bsnm{Stolyar},~\bfnm{A.~L.}\binits{A.~L.}}
(\byear{1992}).
\btitle{On the ergodicity of random processes that describe the
functioning of open queueing networks}.
\bjournal{Probl. Inf. Transm.}
\bvolume{28}
\bpages{3--26}.
\end{barticle}
%
\bptok{imsref}%
\endbibitem

\bibitem{Shah2}
%
\begin{barticle}[mr]
\bauthor{\bsnm{Shah},~\bfnm{Devavrat}\binits{D.}} \AND
\bauthor{\bsnm{Shin},~\bfnm{Jinwoo}\binits{J.}}
(\byear{2012}).
\btitle{Randomized scheduling algorithm for queueing networks}.
\bjournal{Ann. Appl. Probab.}
\bvolume{22}
\bpages{128--171}.
\bid{doi={10.1214/11-AAP763}, issn={1050-5164}, mr={2932544}}
\end{barticle}
%
\bptok{imsref}%
\endbibitem

\bibitem{Shah}
%
\begin{barticle}[mr]
\bauthor{\bsnm{Shah},~\bfnm{Devavrat}\binits{D.}} \AND
\bauthor{\bsnm{Wischik},~\bfnm{Damon}\binits{D.}}
(\byear{2012}).
\btitle{Log-weight scheduling in switched networks}.
\bjournal{Queueing Syst.}
\bvolume{71}
\bpages{97--136}.
\bid{doi={10.1007/s11134-012-9306-x}, issn={0257-0130}, mr={2925792}}
\end{barticle}
%
\bptok{imsref}%
\endbibitem

\bibitem{Tass}
%
\begin{barticle}[mr]
\bauthor{\bsnm{Tassiulas},~\bfnm{Leandros}\binits{L.}} \AND
\bauthor{\bsnm{Ephremides},~\bfnm{Anthony}\binits{A.}}
(\byear{1992}).
\btitle{Stability properties of constrained queueing systems and
scheduling policies for maximum throughput in multihop radio networks}.
\bjournal{IEEE Trans. Automat. Control}
\bvolume{37}
\bpages{1936--1948}.
\bid{doi={10.1109/9.182479}, issn={0018-9286}, mr={1200609}}
\end{barticle}
%
\bptok{imsref}%
\endbibitem

\bibitem{Wischik}
%
\begin{bincollection}[auto:STB|2014/08/04|07:23:14]
\bauthor{\bsnm{Wischik},~\bfnm{Damon}\binits{D.}}
(\byear{2010}).
\btitle{Queueing theory for switched networks}.
In \bbooktitle{ICMS Workshop on Stochastic Processes in Communication
Networks for Young Researchers (Edinburgh), June 2010}.
\bnote{Available at \url{http://www.cs.ucl.ac.uk/staff/ucacdjw/Talks/netsched.html}.}
\end{bincollection}
%
\bptok{imsref}%
\endbibitem
\end{thebibliography}
\end{document}